\documentclass{article}

\usepackage{amsmath,amssymb,mathtools,amsthm}
\usepackage{xcolor,graphicx,float}
\usepackage[utf8]{inputenc}
\DeclareUnicodeCharacter{00A0}{ }
\usepackage[a4paper]{geometry}

\newcommand{\ee}{\varepsilon}
\newcommand{\G}{\mathbb{G}}

\newcommand{\SFL}{\mathrm{SFL}}
\newcommand{\RFL}{\mathrm{RFL}}

\newcommand{\Ls}{(-\Delta)^s}

\newcommand{\Green}{\mathcal{G}}

\newcommand{\cC}{{\mathcal C}}
\newcommand{\cM}{{\mathcal M}}

\DeclareMathOperator{\sign}{sign}
\DeclareMathOperator{\supp}{supp}

\newcommand{\Dirac}{\delta}
\newcommand{\one}{\mathbf 1}
\newcommand{\ulim}{\underline u}

\newcommand{\dx}{\mathrm{d}x}

\newtheorem{theorem}{Theorem}[section]
\newtheorem{proposition}{Proposition}[section]
\newtheorem{corollary}{Corollary}[section]
\newtheorem{lemma}{Lemma}[section]

\theoremstyle{definition}
\newtheorem{definition}{Definition}[section]
\newtheorem{myquestion}{Question}

\newtheorem{remark}{Remark}[section]

\usepackage[hidelinks]{hyperref}
\usepackage{cleveref}
\usepackage{enumerate}

\numberwithin{equation}{section}

\begin{document}
\title{The fractional Schr\"odinger equation \\ with singular potential and measure data}

	\author{%
		D. Gómez-Castro%
			\thanks{Instituto de Matem\'atica Interdisciplinar, Universidad Complutense de Madrid. \url{dgcastro@ucm.es}} %
		\and
		J.L. Vázquez%
			\thanks{Departamento de Matem\'aticas, Universidad Autónoma de Madrid. \url{juanluis.vazquez@uam.es}}
	}
\maketitle

\begin{abstract}
	We consider the steady fractional Schr\"odinger equation $L u + V u = f$ posed on a bounded domain $\Omega$; $L$ is an integro-differential operator, like the usual versions of the fractional Laplacian $(-\Delta)^s$;  $V\ge 0$ is a potential  with possible singularities, and the right-hand side are integrable functions or Radon measures.
	We reformulate the problem via the Green function of $(-\Delta)^s$ and prove well-posedness for functions as data.
	If $V$ is bounded or mildly singular a unique solution of $(-\Delta)^s u + V u = \mu$ exists for every Borel measure $\mu$.
	On the other hand, when $V$ is allowed to be more singular, but only on a finite set of points, a solution of $(-\Delta)^s u + V u = \delta_x$, where $\delta_x$ is the Dirac measure at $x$, exists if and only if $h(y) = V(y) |x - y|^{-(n+2s)}$ is integrable on some small ball around $x$.
	We prove that the set $Z = \{x \in \Omega : \textrm{no solution of } (-\Delta)^s u + Vu = \delta_x \textrm{ exists}\}$ is relevant in the following sense: a solution of  $(-\Delta)^s u + V u = \mu$ exists if and only if $|\mu| (Z) = 0$. Furthermore, $Z$ is the set points where the strong maximum principle fails, in the sense that for any bounded $f$ the solution of $(-\Delta)^s u + Vu = f$ vanishes on $Z$.
	
\end{abstract}

\



\setcounter{tocdepth}{2}
\tableofcontents

\vspace{.5cm}

\newpage

\section{Introduction and outline of results}

We study equations of the form
 \begin{equation}	
 \tag{P$_V$} 	
 \label{eq1}
 \begin{dcases}
L u + V u = f & \Omega, \\
 u = 0 & \partial \Omega \ (\textrm {resp. } \Omega^c)\,,
 \end{dcases}
 \end{equation}
 where  $L$ is an integro-differential operator, we are thinking of the usual Laplacian or one the usual fractional Laplacians $\Ls$ posed on a bounded domain $\Omega$ of $\mathbb R^n$, where $n \ge 3$ and $0<s\le 1$.  $V$ (the potential) is a nonnegative Borel measurable function. In the paper we will  assume Dirichlet boundary conditions to focus on the most relevant setting, but this is in no way essential. We recall that for nonlocal operators boundary conditions are usually replaced by exterior conditions. There are excellent references to nonlocal elliptic equations, both linear and nonlinear, see e.g. \cite{BucurValdi, CabreSire2014, CaffSilv2007, FelKassV2015, Ros-Oton2016}.

 We have recently studied Problem \eqref{eq1} in \cite{diaz+g-c+vazquez2018}, in the case where $L$ is the so-called restricted fractional Laplacian on a bounded domain. The problem was solved for all locally integrable potentials $V\ge 0$ and all right-hand data $f$ in the weighted space $L^1(\Omega, \mbox{dist} (\cdot, \Omega^c)^s)$, which turns out to be optimal for existence and uniqueness of so-called very weak solutions.

 The aim of the present paper is to extend the theory in two directions. Firstly, we want to consider a general class of operators for which a common theory can be constructed. This part of the paper encounters no major obstacles once the proper functional setting is found involving the properties of the Green functions.

 Secondly, we want to extend the theory from integral functions $f$ to Radon measures $\mu$. In doing that we will find a delicate existence problem   when the potential $V$ is singular and $\mu$ is a measure, since $V$ and $\mu$ may be incompatible. We want to understand this difficulty by  characterizing and describing the situation  when nonexistence happens. We start by introducing  a suitable concept of  generalized solution \rm obtained from  natural approximations. This kind of approximation process gives rise to candidate solutions  often known as \sl SOLA solutions \rm or \rm limit solutions when they are admissible solutions.

 Finally, we describe what happens to the  approximations in case of nonexistence: the limit solves the modified problem corresponding to a reduced measure $\mu_r$ instead of $\mu$. Reduced measures are compatible with $V$ and the solution to the problem with $V$ and $\mu_r$ is a kind of closest admissible problem to the original one.

\paragraph{Redefinition of the problem for general operators.} We will follow a trend that has been successfully used in the recent literature on elliptic and parabolic equations involving fractional Laplacians, cf. \cite{Bonforte+Sire+Vazquez2015,Bonforte+Vazquez2016,bonforte+figalli+vazquez2018} which consists in recalling that the main  fractional operators that appear in the literature  have a Green operator $\Green : f \mapsto \overline u$, where $\overline u$ is the unique solution of the inverse problem
 \begin{equation}
 \label{eq:Laplace with L}
 \tag{P$_0$}
\begin{dcases}
\Ls \overline  u = f& \Omega, \\
\overline u = 0 & \partial \Omega \ (\textrm {resp. } \Omega^c).
\end{dcases}
\end{equation}
This solution is given by
\begin{equation}
	\tag{G}
	\label{eq:integral expression of Green}
\overline u(x)=	\Green (f) (x) = \int_ \Omega \G(x,y) f (y) dy.
\end{equation}
The important point is that $\Green$ has very good functional properties acting on classes of continuous or $L^p$ data $f$.
We will list below in \Cref{sec:hypothesis Green} the specific assumptions that determine the class of operators $\Green$ that we can consider.
In \Cref{sec:examples laplacians} we make sure that main examples of fractional operators are included. The Green operator approach is quite efficient and leads us to propose a suitable definition of solution.

\begin{definition} A dual solution of \eqref{eq1} for data $f\in L^1(\Omega)$ is a function $u \in L^1 (\Omega)$  such that
		\begin{subequations}
		\label{eq:fixed point formulation}
		\begin{gather}
		Vu \in L^1 (\Omega) \\ 
		u = \Green ( f - Vu )
		\end{gather}
	\end{subequations}	
\end{definition}
In \Cref{sec:definitions} we show how this definition matches previous notions: very weak solutions and weak-dual solutions. See in this respect previous proposals like those of \cite{Bonforte+Sire+Vazquez2015} and \cite{bonforte+figalli+vazquez2018} dealing with nonlinear parabolic problems and elliptic problems, resp.

\medskip

\subsection{Outline of results}

We state the main contributions.

\paragraph {Results for operators without potentials.}
\Cref{sec:L1 theory for G} contains general facts about the action of operators $\Green$ with attention to covering the examples of operators introduced in \Cref{sec:hypothesis Green}.   \normalcolor Due to \eqref{eq:regularization} we show by duality that $\Green: \cM (\Omega) \to L^1 (\Omega)$ and, hence, \eqref{eq:fixed point formulation} can be extend the theory to the case where $f \in L^1 (\Omega)$ is replaced by a measure $\mu \in \cM(\Omega)$.
In \Cref{sec:definitions} we discuss the definition of dual, weak-dual and very weak solutions for the problem with and without a potential $V$.

\paragraph {Results for operators with bounded potentials.}  \Cref{sec:existence for V bounded} presents the general existence and uniqueness theory under the assumptions that $V$ is bounded while $f$ is merely integrable. In other words, we construct the operator $\Green_V$ for $V \in L^\infty$. The solution is constructed as a fixed point.

\paragraph{Uniqueness for general potentials.} In  \Cref{sec:uniqueness of Schrodinger} we prove that, under some assumptions on $\Green$, there exists at most one solution of \eqref{eq:fixed point formulation}. When it exists, it will obtained  as $\Green_V (\mu)$. The difficult question is whether this solution exists in the sense of our definitions. In Section \ref{sec.exist.L1L1} we prove uniqueness for $V\ge 0$ and $f$ merely integrable.

\paragraph {Results for integrable potentials and data.} In  \Cref{sec.exist.L1L1} we deal with the case: $f, V \in L^1 (\Omega)$. In paper \cite{diaz+g-c+vazquez2018} we were interested in understanding the effect of a singularity of $V$ at the boundary, and so we chose $V \in L^1_{loc} (\Omega)$, $f \mathrm{d} (x, \Omega^c)^s \in L^1 (\Omega)$ and we also studied the Restricted Fractional Laplacian ($(-\Delta)^s_\RFL$) as operator. Under those circumstances we proved existence in all cases,  because we restricted to functions. Our approach of double limit used in that paper will still work here, for general $\Ls$, when $(f,V) \in L^1 (\Omega) \times L^1_+(\Omega)$.

\medskip

\noindent {\bf Interaction of singular potentials and measures.} We now turn our attention to the existence theory when the integrable function $f$ is replaced by a  measure $\mu$. The problem lies in the interaction of the measure with an unbounded potential $V\ge 0$. We find an obstacle to existence if $V$ is too singular at points where the measure has a  discrete component.

In order to focus on the main obstacle, we consider only potentials $V\ge 0$ with isolated singularities. The precise condition is as follows:  $V$ will be singular, at most, at a finite set $S \subset \Omega$  and
	\begin{equation}	
		V : \Omega  \to [0, + \infty] \textrm{ is measurable and } L^\infty\left( \Omega \setminus \bigcup_{x \in S} B_\rho (x) \right) \textrm{ for all } \rho > 0, \nonumber
		\tag{V1}
		 \label{eq:V singular 0}
	\end{equation}
Notice that we specify no particular rate of blow-up at the points of $S$.

In Section \ref{sec:existence for V singular at 0} we introduce the approximation method by means of bounded regularized potentials $V_k= V \wedge k$, that will lead us to the existence of a well-defined limit, that we call the Candidate Solution Obtained as Limit of Approximation (CSOLA). This works for all Radon measures $\mu$ as right-hand side. In the case where $f\in L^p (\Omega)$ we prove existence of a dual solution as a limit of $\Green_{V_k} (f)$, and we study the limit operator $\Green_V$.

\paragraph{Characterizing solvability and describing non-existence} In Section \ref{sec.nonex} we address the question of  nonexistence when  $\mu$ and $V$ turn out to be incompatible.
As the most representative instance, we first address the case where $\mu$ is a point mass and describe what happens when no solution exists in the form of {\sl concentration phenomenon} for $Vu$. In that case, it happens that if $u_k$ is the sequence of approximate solutions, then
\begin{equation}
u_k\to 0, \quad     \mbox{and} \quad  V_k u_k\to \delta_{x_0}.
\end{equation}
This allows to introduce the set $Z$ of incompatible points
	\begin{equation}
	Z = \{  x \in \Omega :  \textrm{ there is no dual solution of \eqref{eq:fixed point formulation} when } \mu = \delta_x   \}.
	\end{equation}
We also have the concept of {\sl reduced measure}. For a measure with support intersecting $Z$,  the obtained CSOLA is not a solution of \eqref{eq1} with data $\mu$, but it is the solution corresponding  to a {\sl reduced measure} associated to $\mu$, $V$ and $G$, which is given by
\begin{equation}
		\mu_r = \mu - \sum_{x \in Z} \mu (\{x\}) \delta_x.
\end{equation}
The notion of  reduced measure was introduced by Brezis, Marcus and Ponce \cite{Brezis2004a,brezis+marcus+ponce2007} in the study of the nonlinear Poisson equation $-\Delta u +g(u) =\mu$. See precedents in \cite{Brezis2003,vazquez_1983}. A excellent general reference is \cite{Ponce2016}.

\paragraph{Properties of the solution operator when $V$ is singular.} We study the limit operator  $\widetilde \Green_V : \cM (\Omega) \to L^1 (\Omega)$ that we call the CSOLA operator. This leads to the questions of the next paragraph.

\paragraph{$Z$ and the loss of the strong maximum principle.} In  Section \ref{sec.Z} address the problem of better understanding $Z$.
First, we  relate the solvability of the problem with a delta measure at a point $x_0\in S$ with the set of points where the Strong Maximum Principle does not hold for solutions with bounded data.
In this investigation we follow ideas developed by Orsina and Ponce for the classical Laplacian \cite{Orsina2018}.
More precisely, we show that a set of universal zeros is precisely the set of incompatible points, i.e.
\begin{equation}
	Z = \{  x \in \Omega : \Green(f) (x) = 0 \quad  \forall f \in L^\infty (\Omega)  \}
\end{equation}
This can be easily explained in \Cref{thm:Z V decomposed} by the fact that the kernel $\G_V$ of the operator $\Green_V$ vanishes:
\begin{equation}
	x \in Z \iff \G _V(x, y) = 0, \quad \textrm{ a.e. } y \in \Omega.
\end{equation}
In fact the kernel $\G_V$ induces an operator $\widetilde \Green_V$ which extends $\Green_V$, but does not necessarily give solutions of \eqref{eq:fixed point formulation}. Furthermore,
\begin{equation}
	\Green_V (\delta_x) \textrm{ is defined } \iff \widetilde \Green_V (\delta_x) \ne 0.
\end{equation}
The existence of this set $Z$ set is caused by $V$.

Work in this direction for the classical Laplacian using capacity can be found in \cite{Rakotoson2018}.

\paragraph{Complete characterization of $Z$.} Finally, under our assumption that $V$ has only isolated singular points, $Z$ is completely characterized in \Cref{thm:Z depending on V} by the condition
\begin{equation}
	x \notin Z \iff \int_{ B_\rho (x) } \frac{V(y)}{|x-y|^{n+2s}}dy <+ \infty  \quad \textrm{ for some } \rho>0 \textrm{ small enough}.
\end{equation}
Notice that, naturally, $Z \subset S$.
\medskip

\paragraph {Comments.}
Our results on singular potentials extend to fractional operators the results in \cite{Orsina2018} when $S = \{ x : V(x) = +\infty \} $ is a discrete set. However, our approach to the proof is completely different. We prove a solution exists if and only if it is the limit of approximating sequences corresponding to a cut-off $V_k = V \wedge k$, and we carefully study this limit. We explain what the limit is in all cases. Actually, we have seen that in the case of nonexistence, a degenerate situation happens where a part of the singular data $\mu$ remains concentrated as the singular part of the limit of the potential term $Vu$.

\subsection{Basic hypothesis on $\Green$}
\label{sec:hypothesis Green}

We list the properties that we will use in the study.
All of them are satisfied by the Green operators that are inverse to the usual Laplacians with zero Dirichlet boundary or external conditions.

(i) $\Green$ is symmetric and self-adjoint in the sense that
\begin{equation}
\tag{G1}
\label{eq:G is symmetric}
\G (x,y) = \G(y,x).
\end{equation}

(ii)  We assume $n \ge 3$ and we have the estimate
\begin{equation} 	\tag{G2}
\label{eq:estimate for G}
\G(x,y) \asymp \frac{1}{|x-y|^{n-2s}} \left( \frac{\Dirac(x) \Dirac(y)}{|x-y|^{2}} \wedge 1 \right)^\gamma.
\end{equation}
We call $0 < s \le 1$ the fractional order of the operator by copying from what happens for the standard of  fractional Laplacians, while  $0 < \gamma \le 1$ distinguishes between the different known cases fractional Laplacians via the boundary behaviour.

In some cases it could be sufficient to require that for every compact $K \Subset \Omega $ we have
\begin{equation}
0 < \frac{ c_K }{|x-y|^{n-2s}} \le \G (x,y) \le \frac{ C_K } {|x-y|^{n-2s}},
\end{equation}
but this not generally used.

(iii) Furthermore, we need positivity in the sense that
\begin{equation}
\tag{G3}
\label{eq:coercivity}
\int_ \Omega f \Green (f) \ge 0  \qquad \forall f \in L^2 (\Omega)
\end{equation}
The hypothesis above often follows from the stronger property of coercivity that holds for the standard versions of fractional Laplacian in forms like
\begin{equation}
\| (-\Delta)^{\frac s 2} u \|_{L^2} \le \int_ \Omega u \Ls u.
\end{equation}
Putting  $f = L u$  so that $u = \Green (f)$, we get
\begin{equation}
\| \Green(f)\|^2 \le \int_ \Omega \Green (f) f.
\end{equation}

(iv) Lastly, we assume $\Green$ is regularizing in the  sense that
\begin{equation}
\tag{G4}
\label{eq:regularization}
\Green: L^\infty (\Omega) \to \cC (\overline \Omega).
\end{equation}
Conditions for this property to hold are well-known for the main fractional operators (see, e.g., \cite{Ros-Oton2016} and the references therein). In the case of the most common choice,  Restricted Fractional Laplacian (RFL) we refer to \cite{Ros-Oton2014}).  For the Spectral Fractional Laplacian (SFL) a convenient reference is \cite{Caffarelli+Stinga2016}.

Interior regularity is usually higher (see \cite{Cozzi2017}). A general reference to fractional Sobolev spaces, embeddings and related topics if, or instance, \cite{DiNezza2012}.

\subsection{Usual examples of admissible operators}
\label{sec:examples laplacians}

\subsubsection{The classical Laplacian $-\Delta$}
In this case it is known
\begin{enumerate}
	\item \eqref{eq:estimate for G} holds with $s = 1$ and $\gamma = 1$.
	\item \eqref{eq:coercivity} is well known.
	\item The regularization \eqref{eq:regularization} is a classical result. See, e.g., \cite{Evans1998,Gilbarg+Trudinger2001}.
\end{enumerate}

\subsubsection{Restricted Fractional Laplacian $(-\Delta)^s_\RFL$}
This operator is given by
\begin{equation}
\label{eq:RFL}
(-\Delta)^s_\RFL  u (x) =c_{n,s} \int_{\mathbb R^n} \frac{u(x)-u(y)}{|x-y|^{n+2s}} dy
\end{equation}
where $u$ is extended by $0$ outside $\Omega$.
In this case it is known
\begin{enumerate}
	\item \eqref{eq:estimate for G} holds with $0 < s < 1$ and $\gamma = s$
	\item \eqref{eq:coercivity} since, for $f \in L^\infty (\Omega)$
	\begin{equation}
	\int_ \Omega f \Green (f) = \int_ \Omega (- \Delta)^s (\Green (f)) f = \int_ \Omega | (-\Delta)^{s/2} (\Green(f)) |^2 \ge 0.
	\end{equation}
	For the remaining functions we apply density.
	\item The regularization \eqref{eq:regularization} is proven via Hörmander theory. See, e.g. \cite{Grubb2015,Ros-Oton2014}.
\end{enumerate}

\subsubsection{Spectral Fractional Laplacian $(-\Delta)^s_\SFL$}
This operator is given by
\begin{equation}
(-\Delta)^s_\SFL  u (x) = \sum_{i=1}^{+\infty} \lambda_i^{s} u_i \varphi_i (x)
\end{equation}
where $(\varphi_i,\lambda_i )$ is the spectral sequence of the Laplacian with homogeneous Dirichlet boundary condition and $u_i = \int_ \Omega u \varphi_i$.
In this case it is known
\begin{enumerate}
	\item \eqref{eq:estimate for G} holds with $0 < s < 1$ and $\gamma = 1$
	\item \eqref{eq:coercivity} since, for $f \in H^1 (\Omega)$
	\begin{equation}
	\int_ \Omega f \Green (f) = \sum_{i=1}^{+\infty}\lambda_i^s f_i^2 \ge 0 .
	\end{equation}
	\item The regularization \eqref{eq:regularization} can be found in \cite{Caffarelli+Stinga2016}.
	
\end{enumerate}

\subsubsection{Other examples}
There are a number of other operators that can be considered like the Censored (or Regional) Fractional Laplacian which is described in many references, like \cite{bonforte+figalli+vazquez2018}.

\section{The elliptic equation without potential}
\label{sec:L1 theory for G}

\subsection{Immediate properties}
\begin{subequations}
	The following are immediate consequence of the kernel representation
\begin{lemma}
	Assume that $\G (x,y) \ge 0$.  Then, the Green operator \eqref{eq:integral expression of Green} is monotone in the sense that
	\begin{gather}
		\label{eq:monotonicity of Green}
		0 \le f \in L^\infty (\Omega) \implies 0 \le \Green (f).
	\end{gather}
	If, furthermore, \eqref{eq:G is symmetric} then \eqref{eq:integral expression of Green} is self-adjoint:
	\begin{gather}
	\label{eq:Green self adjoint}
	  \int_ \Omega \Green (f) g = \int_ \Omega f \Green (g) \qquad \forall f , g \in L^\infty (\Omega) .
	\end{gather}
\end{lemma}
\end{subequations}
\begin{proof}
	For the monotonicity we simply take into account that $\G \ge 0$ and therefore $\G (x,y) f(y) \ge 0$. To show that it is self-adjoint we compute explicitly
	\begin{align}
		\int_ \Omega \Green(f) (x) g(x) dx &= \int_ \Omega \left( \int_ \Omega \G(x,y) f(y) dy \right)  g(x) dx \nonumber\\
		&=  \int_ \Omega \int_ \Omega \G(x,y) f(y) g(x) dy dx \nonumber\\
		&= \int_ \Omega \int_ \Omega \G(y,x) f(y) g(x) dx dy \nonumber\\
		&= \int_ \Omega f(y) \left( \int_ \Omega \G(y,x) g(x) dx \right) dy \nonumber \\
		&= \int_ \Omega f(y) \Green(g) (y) dy
	\end{align}
	This completes the proof.
\end{proof}

\subsection{Regularization}

\label{sec:regularization}

\begin{theorem}
	\label{thm:regularization}
	If $f \in L^p (\Omega)$ then $\Green(f) \in L^{ q }(\Omega)$ for all $1 \le q < Q(p) = \frac{n}{n-2s}p$. Furthermore $\Green: L^p (\Omega) \to L^{q} (\Omega)$ is continuous.
\end{theorem}

Our aim is to apply the Riesz-Thorin interpolation theorem (see, e.g., \cite{Triebel}).
\begin{theorem}[Riesz-Thorin convexity theorem]
	Let $T$ be a linear operator such that
	\begin{align}
		T&:L^{p_i} (\mathbb R^n) \to L^{q_i} (\mathbb R^n), \qquad i = 0,1
	\end{align}
	 is continuous for some $1 \le p_0 ,p_1 , q_0, q_1 \le +\infty$ and let, for $\theta \in (0,1)$ define
	\begin{equation}
		\frac{1}{p_\theta} = \frac{1- \theta}{p_0} + \frac{\theta}{p_1}, \qquad \frac{1}{q_\theta} = \frac{1- \theta}{q_0} + \frac{\theta}{q_1}.
	\end{equation}
	Then
	\begin{equation}
		T : L^{p_\theta} (\mathbb R^n) \to L^{q_\theta} (\mathbb R^n)
	\end{equation}
	is continuous. Furthermore
	\begin{equation}
		\| T \|_{\mathcal L (L^{p_\theta}  , L^{q_\theta})} \le 	\| T \|_{\mathcal L (L^{p_0} , L^{q_0})}^{1-\theta}\| T \|_{\mathcal L (L^{p_1}  , L^{q_1})}^{\theta}.
	\end{equation}
\end{theorem}

\begin{proposition}
	\label{thm:regularization from L1}
	Let $f \in L^1 (\Omega)$. Then $\Green (f) \in L^{q} (\Omega)$ for $1 \le q < Q(1) = \frac{n}{n-2s}$ and the map $\Green: L^1 (\Omega) \to L^q (\Omega)$ is continuous.
\end{proposition}

We split the proof in some lemmas. The two first lemmas can be found in \cite{bonforte+figalli+vazquez2018} and are given here for the reader's convenience
\begin{lemma}
	\begin{equation}
	\int_ \Omega |\G (x,y)|^{q} dy \le C, \textrm{ where } 1 \le q <  \frac{ n } { n - 2s}
	\end{equation}
	and  $C$ does not depend on $x \in \Omega$.
\end{lemma}
\begin{proof}
	We take $R$ large enough so that $\Omega \subset B_R (x)$ for every $x \in \Omega$. We have that
	\begin{align}
		\int_ \Omega |\G (x,y)|^{ q } dy &\le C \int_\Omega |x-y|^{(2s-n)q} dy \nonumber  \\
			&\le \int_{B_R (x)} |x-y|^{(2s-n)q} dy \nonumber \\
			&\le C \int_0^R r^{(2s-n)q} r^{n-1}dr
			 \le C
	\end{align}
	if $(2s-n)q + n > 0$. In other words if $q < \frac{n}{n-2s}$.
	This completes the proof.
\end{proof}
Through duality it is mediate that
\begin{lemma}
	$\Green : L^{q'} (\Omega) \to L^\infty (\Omega)$ is continuous for all $1 \le q < Q(1)$.
\end{lemma}
\begin{proof}
	Through Hölder's inequality
	\begin{equation}
		|\Green(f) (x)| = \left| \int_ \Omega \G(x,y) f(y) \right| \le \int_ \Omega \G (x,y) |f(y)| \le \| \G(x, \cdot) \|_{q} \| f \|_{q'} \le C \| f \|_{q'} .
	\end{equation}
	and this holds uniformly on $x \in \Omega$.
\end{proof}

We can now prove the theorem.
\begin{proof}[Proof of \Cref{thm:regularization}]
	Due to the Riesz-Thorin interpolation theorem since $\Green: L^1 (\Omega) \to L^{\gamma} (\Omega)$ with $1 \le \gamma < Q(1)$ and $\Green : L^\infty (\Omega) \to L^{\infty} (\Omega)$ then
	$\Green: L^p (\Omega) \to L^{\gamma p } (\Omega)$. Therefore $\Green: L^p (\Omega) \to L^{q} (\Omega)$ for
	where $1 \le q < p Q(1) = Q(p)$.
\end{proof}

\begin{remark}
	\label{rem:eigenfunctions are continuous}
		Notice that this immediately implies that eigenfunctions are in $\cC(\Omega)$. Indeed, let
		\begin{equation}
			\underline Q (1) = \frac{1 + Q(1)}{2} \in (1, Q(1)), \qquad \underline Q(p) = p \underline Q(1) \in (p, \underline Q(p) ).
		\end{equation}
		$u = \lambda \Green (u)$. If $u \in L^1 (\Omega)$ then $\Green (u) \in L^{\underline Q(1)} (\Omega)$ and so $u \in L^{\underline Q(1)} (\Omega)$. Analogously $u \in L^{\underline Q^n(1)} (\Omega)$ for every $n \ge 1$. After a finite number of iterations we have $\underline Q^n(1) > (Q(1))'$. Therefore $u \in L^\infty (\Omega)$. But then $u = \lambda \Green (u) \in \mathcal C(\Omega)$.
\end{remark}

\subsection{Dunford-Pettis property of $\Green$}
\label{sec:Dunford-Pettis}
The aim of this section is to prove that
\begin{theorem}
	\label{thm:Dunford-Pettis for G}
	We have that, for any $0 < \beta < \frac{2s}n$
	\begin{equation}
	\int_A |\Green (f)| \le C |A |^\beta \| f \|_{L^1 (\Omega)}, \qquad \forall f \in L^1 (\Omega).
	\end{equation}
	for some $C > 0$.
	In particular,
	for every bounded sequence $f_n \in L^1 (\Omega)$ the sequence $\Green (f_n)$ is equiintegrable. In particular, there exists a weakly convergent subsequence $\Green (f_{n_k}) \rightharpoonup u$ in $L^1 (\Omega)$.
\end{theorem}

For this we introduce the following auxiliary estimate
\begin{lemma}
	We have that
	\begin{equation}
	\label{eq:Green of indicator function}
	\| \Green (\one_A) \|_{L^\infty} \le C |A|^\beta, \qquad \textrm{ for any  } 0 < \beta < \frac{2s}{n}, \ \forall A \subset \Omega .
	\end{equation}
	where $C$ depends on $\beta$ but not on $A$.
\end{lemma}

\begin{proof}
	We have that $\Green : L^{p} (\Omega) \to L^\infty (\Omega)$ for $p > Q(1)'$. Hence
	\begin{equation}
		\| \Green (\one_A) \|_{L^\infty} \le C \| \one_A \|_{L^{p} (\Omega)} = \left(  \int_A 1^{{p}}  \right)^{1 / p}  = C |A|^{1 / p}.
	\end{equation}
	Taking $\beta = \frac 1 {p}$ we complete the proof.
\end{proof}

\begin{proof}[Proof of \Cref{thm:Dunford-Pettis for G}]
	We prove that $\Green (f)$ satisfies
	\begin{align}
	\int_A |\Green(f)| &= \int_{ \Omega } |\Green (f) | \one_A \le \int_ \Omega \Green(|f|) \one_A
	= \int_{ \Omega } |f| \Green (\one_A) \le \| f \|_{L^1} \| \Green(\one_A) \|_{L^\infty} \nonumber \\
	& \le C |A|^\beta  \| f \|_{L^1 (\Omega)}.
	\end{align}
	This completes the proof.
\end{proof}

\begin{remark}
	Using Marcinkiewicz spaces the results in \Cref{sec:regularization,sec:Dunford-Pettis} can be proved with equality in the range of $\beta$. The required information about Marcinkiewicz spaces can be found in  \cite{Benilan+Brezis+Crandall1975}.
\end{remark}

\subsection{Extension of $\Green$ to $\cM (\Omega)$}
To use data in $\mathcal M (\Omega)$ we need the stronger assumptions \eqref{eq:regularization}, which we have not used until now.

We will extend our results by approximation. This philosophy has been applied successfully over the years (see, e.g., \cite{Kuusi+Mingione+Sire2015a} for relevant recent work in the nonlocal case).

\begin{theorem}
	\label{thm:extension of G to measures}
	Let $\Green$ satisfy \eqref{eq:Green self adjoint} and \eqref{eq:regularization}.
	Then, there exists an extension
	\begin{equation}
	\Green : \cM (\Omega) \to L^1 (\Omega).
	\end{equation}
	which is linear and continuous.
		Furthermore, this extension is unique and self-adjoint. The function $u = \Green (\mu)$ is the unique function such that $u \in L^1 (\Omega)$ and
		\begin{gather}
				\label{eq:Laplace vwf}
			\int_ \Omega u \psi = \int_\Omega \Green(\psi) \mathrm{d} \mu, \qquad \forall \psi \in L^\infty_c (\Omega).
		\end{gather}
\end{theorem}

\begin{proof}

	Let $\mu \in \cM (\Omega)$. By density let $f_n \in L^\infty (\Omega)$ such that $f_n \, \dx \rightharpoonup \mu$ and bounded in $L^1 (\Omega)$. Due to \Cref{thm:Dunford-Pettis for G} there exists a subsequence $\Green (f_{n_k})$ converging weakly in $L^1(\Omega)$. Let $u$ be its limit. Furthermore
	\begin{equation}
	\| u \|_{L^1 (\Omega)} \le \liminf_n \| u_{n_k} \|_{L^1 (\Omega)} \le C \liminf \| f_{n_k} \|_{L^1 (\Omega)} = C \| \mu \|_{\cM (\Omega)}.
	\end{equation}
	Due to \eqref{eq:Green self adjoint}
	\begin{equation}
	\int_{ \Omega } u_{n_k} \psi = \int_ \Omega \Green(\psi) f_{n_k} \dx.
	\end{equation}
	Passing to the limit, since $\Green (\psi ) \in \cC (\Omega)$ we deduce

	There is at most one element with this property. If there two $u_1, u_2$ letting $w = u_1-u_2$ we would have
	\begin{equation}
		\int_ \Omega w \psi = 0 , \qquad \forall \psi \in L^\infty (\Omega).
	\end{equation}
	Taking $\psi = \sign_+ w$ we deduce $ w = 0$, so $u_1 = u_2$.
	
	Hence, our definition $\widetilde \Green (\mu) = u$ is consistent.
	
	Linearity. To show continuity we prove boundedness. Let $\mu\in \cM (\Omega)$ we have
	\begin{equation}
	\int_{ \Omega } \widetilde \Green (\mu) \psi =  \int_\Omega \Green (\psi ) \mathrm{d} \mu \le \| \Green (\psi) \|_{\cC} \| \mu \|_{\cM (\Omega)}.
	\end{equation}
	Taking $\psi = \sign (\Green (\mu))$ we deduce
	\begin{equation}
	\| \widetilde \Green (\mu)  \| \le \|\Green (\one_\Omega) \|_{\cC} \| \mu \|_{\cM (\Omega)}.
	\end{equation}
	Furthermore, we have shown that $\widetilde \Green (\psi)$ satisfies \eqref{eq:Laplace vwf}.
\end{proof}

\begin{corollary} For every $\mu \in \cM (\Omega)$
	\begin{equation}
		\int_A \Green (\mu) \le c |A|^\beta \| \mu \|_{\cM}
	\end{equation}
\end{corollary}

\begin{corollary}
	If $\mu_n \rightharpoonup \mu$ weakly in $\cM (\Omega)$ then $\Green( \mu_ n) \rightharpoonup \Green (\mu)$ in $L^1 (\Omega)$.
\end{corollary}

However, the following is stronger:
\begin{proposition}
	\label{prop:G continuous weak star measure to weak L1}
	If $\mu_n \rightharpoonup \mu$ weak-$\star$ in $\cM (\Omega)$ then $\Green( \mu_ n) \rightharpoonup \Green (\mu)$ in $L^1 (\Omega)$.
\end{proposition}

\begin{proof}
	If $\mu_n \rightharpoonup \mu$ weak-$\star$ then $\| \mu_n \|_{\cM} $ is bounded. Thus, $\Green (\mu_n)$ is equiintegrable. Taking a convergent subsequence $\Green (\mu_n) \rightharpoonup \ulim$. Substituting in the formulation
	\begin{equation}
		\int_ \Omega \Green( \mu_n) \psi = \langle \Green (\psi) , \mu_n \rangle .
	\end{equation}
	Passing to the limit
	\begin{equation}
	\int_ \Omega \ulim \psi = \langle \Green (\psi) , \mu \rangle .
	\end{equation}
	Thus $\ulim = \Green (\psi)$. The limit of every subsequence coincides so there is a limit.
\end{proof}

\subsection{Local scaling}

The scaling of
$
	\int_{ B_ \rho } \Green( \mu ) \dx
$
as $\rho \to 0$ will be very significant.

\subsubsection{Away from $\supp \mu$}

\begin{lemma}
	If $\supp \mu \cap B_R (x) = \emptyset$ then
	\begin{equation}
		\int_ { B _\rho (x) } \Green( \mu) \dx \le C (R-\rho)^{2s-n} \rho^n , \qquad \forall \rho < R
	\end{equation}
\end{lemma}
\begin{remark}
	Notice is the natural behaviour at a Lebesgue point since it implies that
	\begin{equation}
		\limsup_{ \rho \to 0 } \frac{1}{|B_\rho|} \int_{ B_\rho} \Green (\mu) \le  C R^{2s-n}
	\end{equation}
\end{remark}

\subsubsection{The sequence $\Green (\one_{B_\rho})$}
Our aim is to show
\begin{proposition}
	Let $x_0 \in \Omega$ and $B_\rho = B_\rho (x_0)$. The following hold
	\begin{subequations}
		\label{eq:G one rho}
	\begin{align}
		\label{eq:G one rho at 0}
		\frac{ \Green (\one_{B_\rho}) }{ \rho^{2s} }(x_0) & \ge c > 0\\
		\label{eq:G one rho L1}
		\frac{ \Green (\one_{B_\rho}) }{ \rho^{2s} } &\to 0   \qquad L^1( \Omega ) \\
		\label{eq:G one rho Linf}
		\frac{ \Green (\one_{B_\rho}) }{ \rho^{2s} } &\rightharpoonup 0  \qquad L^\infty( \Omega )\textrm{-weak-}\star \\
		\label{eq:G one rho pointwise}
		\frac{ \Green (\one_{B_\rho}) }{ \rho^{2s} } &\to 0   \qquad \textrm{ pointwise in } \Omega \setminus \{x_0\}. \\
		\label{eq:G one rho measure}
		\int_ \Omega  \frac{ \Green (\one_{B_\rho}) }{ \rho^{2s} } \mathrm{d} \mu  & \to 0 \qquad \textrm{ for every } \mu \in \cM (\Omega) \textrm{ such that } \mu (\{x_0\}) = 0.
	\end{align}
	\end{subequations}
\end{proposition}

\begin{proof}
	\begin{align}
		0 \le \Green (\one_{B_\rho}) (x) \le \int_ {B_\rho} \G(x,y) dy \le C \int_ {B_\rho} |x-y|^{2s-n} dy \le C \int_{ B_\rho } |y|^{2s - n}dy = C \rho^{2s}.
	\end{align}
	Furthermore, at $x = x_0$ this inequality hold in reverse order hold (except for the first), and \eqref{eq:G one rho at 0} is proven.
	Therefore $\| \Green (\one_{B_\rho})  \|_{L^\infty}$ is bounded. Furthermore
	\begin{equation}
		\int_ \Omega \frac{ \one_{B_\rho} }{\rho^{2s}} = C \rho^{n-2s} \to 0 .
	\end{equation}
	Therefore, due the strong continuity \eqref{eq:G one rho L1} is proven. But then the limit coincides with the weak-$\star$ limit in $L^\infty$, so \eqref{eq:G one rho Linf} is proven. For $x \neq x_0$ we have the sharper estimate, for $\rho < |x-x_0|$
	\begin{align}
		 \frac{ \Green (\one_{B_\rho}) (x)}{\rho^{2s}} \le C \rho ^{-2s} \int_ {B_\rho} |x-y|^{2s-n} dy \le C (|x-x_0| - \rho)^{2s-n} |\rho|^{n-2s} \to 0.
	\end{align}
	
	To prove \eqref{eq:G one rho measure} we assume first that $\mu \ge 0$. When $\mu (\{x_0\}) = 0$ we have that
	\begin{equation}
		0 \le \frac{ \Green (\one_{B_\rho}) (x)}{\rho^{2s}} \le C
	\end{equation}
	and
	\begin{equation}
		\mu \left( \left  \{  x \in \Omega :  \frac{ \Green (\one_{B_\rho}) (x)}{\rho^{2s}} = 0  \right \}  \right ) \le \mu (\{0\}) = 0.
	\end{equation}
	Therefore, the convergence is $\mu$-everywhere. By the Dominated Convergence Theorem we have \eqref{eq:G one rho measure}. When $\mu$ changes sign we reproduce the argument for $\mu^+$ and $\mu^-$ and the result is proven.
\end{proof}

\begin{remark}
	Notice that this is the scaling as $\cC$ function. Obviously $\one_{B_\rho(x)} / |B_\rho (x)|$
	
\end{remark}

\subsubsection{Near the $\supp \mu$}

\begin{proposition}
	\label{thm:local integral of G mu near support}
	Let $\mu \in \cM (\Omega)$. Then
	\begin{equation}
		  \lim_{ \rho \to 0 }\rho^{-2s} \int \limits_{ B_\rho (x) } \Green(\mu) \dx \asymp \mu (\{x\}).
	\end{equation}
\end{proposition}
\begin{proof}
	Assume $\mu(\{x\}) = 0$. Since $\Green$ is self-adjoint
	\begin{equation}
	 	\rho^{-2s} \int_{ B_\rho (x) } \Green(\mu) =\rho^{-2s} \int_{ \Omega } \Green(\mu) \one_{B_ \rho (x)} =  \rho^{-2s} \int_ \Omega  \Green( \one_{B_ \rho (x)} ) d \mu \to 0
	 \end{equation}
	 due to \eqref{eq:G one rho measure}.

	On the other hand let us compute
	\begin{align}
	\int_{ B_\rho (x) } \Green (\Dirac_x) &= \int_{ B_\rho (x) } \G(y, x) \mathrm dy \asymp \int_{ B_\rho (x) } |x-y|^{n-2s} \mathrm dy \nonumber \\
	& = C \int_{0}^\rho r^{2s-n} r^{n-1} \mathrm dr  = C \rho^{2s}.
	\end{align}
	
	Therefore, for a general measure $\mu$ we can decompose
	\begin{equation}
		\Green( \mu ) = \Green \Big(\mu - \mu(\{x\}) \, \Dirac_x \Big) + \mu(\{x\}) \, \Green( \Dirac_x ).
	\end{equation}
	Applying the two preceding parts the result is proven.
\end{proof}

\subsection{Almost everywhere approximation of $\Green (\delta_{x_0})$}

\begin{lemma}
	\label{lem:ae aprox of G delta 0}
	We have that
	\begin{equation}
		\Green \left(   	\frac{\one_{B_{\frac 1 k} (x_0)} }{|B_{\frac 1 k} (x_0)|} \right) \to \Green (\delta_{x_0} ) , \qquad \textrm{ a.e. in } \Omega.
	\end{equation}
	
\end{lemma}

\begin{proof}
	Assume $x \neq x_0$.  For $k_0$ large enough $x \notin \overline {B_{\frac 1 {k_0}} (x_0)}$. Then, due to \eqref{eq:estimate for G}, $\G(x, \cdot) \in L^\infty ( B_{\frac 1 {k_0}} (x_0) )$. Hence $x_0$ is a Lebesgue point of $\G(x, \cdot)$. Due to the Lebesgue integration theorem we have that
	\begin{equation}
	\Green \left(   	\frac{\one_{B_{\frac 1 k} (x_0)} }{|B_{\frac 1 k} (x_0)|} \right) (x) = \frac{1}{|B_{\frac 1 k} (x_0)|} \int_{B_{\frac 1 k} (x_0)} \G (x,y) dy \to \G (x,x_0) = \Green( \Dirac_{x_0} ) (x)
	\end{equation}
	This completes the proof.
\end{proof}

\section{Equivalent definitions of solution}
\label{sec:definitions}

We discuss the definition of dual, weak-dual and very weak solutions for the problem with and without a potential $V$.
\subsection{Problem \eqref{eq:Laplace with L}.}
Brezis introduced the notion of \emph{very weak solution} for the classical case $s=1$ as
\begin{equation}
\int_ \Omega u (-\Delta \varphi) = \int_ \Omega f \varphi, \qquad \forall \varphi \in W^{2,\infty}(\Omega) \cap W^{1,\infty}_0 (\Omega)
\end{equation}
Chen and V\'eron \cite{chen+veron2014} extended this definition to the Restricted Fractional Laplacian as
	\begin{gather}
	\int_ \Omega u (-\Delta)^s_\RFL \varphi = \int_ \Omega f \varphi, \qquad \forall \varphi \in \mathbb X_s
	\end{gather}
where
\begin{equation}
\mathbb X_s = \{  \varphi \in \cC^s( \mathbb R^n ) : \varphi = 0 \textrm{ in } \mathbb R^n \setminus \Omega \textrm{ and } (-\Delta)^s \varphi \in L^\infty (\Omega)  \}
\end{equation}

Letting $\psi = (-\Delta)^s_\RFL \varphi \in L^\infty (\Omega)$, which implies that $\varphi = \Green(\psi)$ this is equivalent to writing
\begin{equation}
\int_ \Omega u \psi = \int_ \Omega f \Green( \psi ) \qquad \forall \psi \in L^\infty (\Omega).
\end{equation}

In some texts (see \cite{bonforte+figalli+vazquez2018}) the authors have used this as a new definition of solution  of \eqref{eq:Laplace with L} for more general operators, and they usually call this \emph{weak dual solution}. It has the advantage that one needs not worry about fancy spaces of test functions, but only on the nature of $\Green$. Furthermore, the treatment of different fractional Laplacians is unified.

Notice that, whenever $\Green(f)$ is defined, since $\Green$ is self-adjoint this is equivalent to
\begin{equation}
	\int_ \Omega u \psi = \int_ \Omega \Green(f) \psi \qquad \forall \psi \in L^\infty (\Omega).
\end{equation}
and since $u$ and $\Green(f)$ are in $L^1(\Omega)$ this is simply
\begin{equation}
	u = \Green(f)
\end{equation}

\subsection{Problem \eqref{eq1}.}
For the Schr\"odinger problem the notion of very weak solution for the classical case was used multiple times in the literature (see, e.g., \cite{diaz+gc+rakotoson+temam:2018veryweak} and the references therein) as
\begin{subequations}
\begin{gather}
	Vu \in L^1 (\Omega), \\
	\int_ \Omega u (-\Delta \varphi) + \int_ \Omega Vu \varphi = \int_ \Omega f \varphi, \qquad \forall \varphi \in W^{2,\infty}(\Omega) \cap W^{1,\infty}_0 (\Omega)
\end{gather}
\end{subequations}
We extended this notion in \cite{diaz+g-c+vazquez2018} to the case $(-\Delta)^s_\RFL$ by using the definition
\begin{subequations}
	\begin{gather}
	Vu \in L^1 (\Omega), \\
	\int_ \Omega u (-\Delta)^s_\RFL \varphi + \int_ \Omega Vu \varphi = \int_ \Omega f \varphi, \qquad \forall \varphi \in \mathbb X_s
	\end{gather}
\end{subequations}
The corresponding notion of weak-dual solution is very naturally
\begin{subequations}
	\begin{gather}
	Vu \in L^1 (\Omega), \\
	\int_ \Omega u \psi + \int_ \Omega Vu \Green( \psi)  = \int_ \Omega f \Green(\psi), \qquad \forall \psi \in L^\infty (\Omega).
	\end{gather}
\end{subequations}
Again, this notion is equivalent to our definition of \emph{dual} solution.


\section{Theory for $(f, V) \in L^1 (\Omega) \times L^\infty_+ (\Omega)$}
\label{sec:existence for V bounded}

Rather complete results are obtained for bounded potentials and integrable data.

\subsection{Existence. Fixed-point approach}

Here we show the following
\begin{theorem}
	\label{eq:existence fixed point}
	Let $f \in L^1 (\Omega)$ and $V \in L^\infty_+(\Omega)$. Then, there exists a solution $u$ of \eqref{eq:fixed point formulation} and it satisfies
	\begin{equation}
		|u| \le \Green (|f|)
	\end{equation}
	Furthermore, if $f \ge 0$ then $u \ge 0$.
\end{theorem}
\begin{proof}
	\textbf{Step 1.} Assume $f \ge 0$.
	We construct the following sequence. $u_0 = 0$, $u_1  = \Green (f) \ge 0$,
	\begin{align}
	u_2 &= \Green \Bigg( \Big(f - V u_{1}\Big)_+ \Bigg), \\
	u_i &= \Green (f - V u_{i-1}), \qquad i > 2.
	\end{align}
	
	\textbf{Step 1a.} We prove that
	\begin{equation}
		u_0 \le u_2 \le u_3 \le u_1.
	\end{equation}
	Clearly $u_0 \le u_1$. Since
	\begin{equation}
	0 \le (f - V u_1 )_+ \le f .
	\end{equation}
	Thus, applying $\Green$, $u_0 \le u_2 \le u_1$. Therefore
	\begin{equation}
		f- V u_1 \le f- V u_2 \le f - V u_0
	\end{equation}
	Applying again $\Green$ we have
	\begin{equation}
		u_2 \le u_3 \le u_1.
	\end{equation}
	
	\textbf{Step 1b.} We show, by induction, that
	\begin{equation}
		u_{2i} \le u_{2i + 2} \le u_{2i + 3} \le u_{2i+1}, \qquad \forall i \ge 0.
	\end{equation}
	The result is true for $i = 0$ by the previous step. Assume the result true for $i$:
	\begin{equation}
		u_{2i} \le u_{2i + 2} \le u_{2i + 3} \le u_{2i+1}
	\end{equation}
	we have that
	\begin{equation}
		f - V u_{2i+1} \le f - V u_{2i+3} \le f - V_{2i+2} \le f - V u_{2i}.
	\end{equation}
	Applying $\Green$ we have that
	\begin{equation}
		u_{2(i+1)} \le u_{2(i+1)+ 2} \le u_{2i+3} \le u_{2i+1}
	\end{equation}
	Repeating the process
	\begin{equation}
		f - Vu_{2i+1} \le f - Vu_{2i+3} \le f - Vu_{2(i+1)+2}  \le f- Vu_{2(i+1)}.
	\end{equation}
	Applying $\Green$
	\begin{equation}
		u_{2(i+1)} \le u_{2(i+1)+ 2} \le u_{2(i+1) + 3} \le u_{2(i+1)+1}
	\end{equation}
	Then the result is true for $i+1$.
	This step is proven.
	
	\textbf{Step 1b.} By the monotone convergence theorem $u_{2i} \nearrow \underline u$ in $L^1 (\Omega)$ where $u_{2i+1} \searrow \overline u$ in $L^1 (\Omega)$. Clearly $0 \le \underline u \le \overline u \le \Green (f) $.
	Since $V \in L^\infty (\Omega)$ then $V u_{2i}$ and $V u_{2i+1}$ also converge in $L^1 (\Omega)$. Since $\Green$ is continuous in $L^1 (\Omega)$ we have
	\begin{gather}
	\overline u = \Green (f - V \underline u), \\
	\underline u = \Green (f - V \overline u).
	\end{gather}
	Therefore $u = \frac 1 2 (\underline u + \overline u)$ is a solution of
	\begin{equation}
	u = \Green (f - Vu).
	\end{equation}

	\textbf{Step 2.} Assume now that $f$ changes sign. We decompose $f = f_+ - f_-$ and solve for each $f_+$ and $f_-$, to obtain $u_1$ and $u_2$. Then, clearly $u = u_1 - u_2$ is a solution of the problem. Furthermore
	\begin{equation}
		|u| = |u_1-u_2| \le u_1 + u_2 \le \Green(f_+) + \Green (f_-) = \Green (|f|).
	\end{equation}
	
	This completes the proof.
\end{proof}

\subsection{Uniqueness}

\begin{theorem}
	\label{thm:uniqueness fixed point}
	$V \in L^\infty (\Omega)$.
	There exists at most one solution $u \in L^1 (\Omega)$ of  \eqref{eq:fixed point formulation}.
\end{theorem}
\begin{proof}
	Let $u_1, u_2$ be two solutions. We proceed as in \Cref{rem:eigenfunctions are continuous}. Its difference $u = u_1-u_2 \in L^1(\Omega)$ satisfies $u = -\Green (Vu)$. Then $V u \in L^1 (\Omega)$ and $u \in L^{\underline Q(1)} (\Omega)$. Repeating this process we deduce that $u \in L^2 (\Omega)$.
	
	Therefore $Vu^2 = -Vu \Green (Vu) \in L^2 (\Omega)$. We deduce
	\begin{equation}
	0 \le \int_ \Omega V u^2 = -  \int_ \Omega Vu \Green (Vu) \le 0.
	\end{equation}
	due to \eqref{eq:coercivity}. Hence $Vu^2 = 0$ and so $V u= 0$. But then $u = - \Green(0) = 0$. The solutions $u_1$ and $u_2$ are equal.
\end{proof}

\subsection{The solution operator $\Green_V$}

\begin{corollary}
	Let $V \in L^\infty_+ (\Omega)$. We consider the  solution operator
	\begin{equation}
		\Green_V : f \in L^1 (\Omega) \mapsto u \in L^1(\Omega)\,,
	\end{equation}
	where $u$ is the unique solution of $u = \Green (f- Vu)$. It is  well-defined, linear and continuous.
\end{corollary}

We leave the easy details to the reader.

\subsection{Equi-integrability independently of $V$}

\begin{theorem}
	\label{thm:Dunford-Pettis for G V with V bounded}
	We have
	\begin{equation}
	\int_A |\Green_V (f)| \le C |A |^\beta \| f \|_{L^1 (\Omega)}, \qquad \forall f \in L^1 (\Omega).
	\end{equation}
	for any $0<\beta < {2s}/{n}$. In particular,
	for every bounded sequence $f_n \in L^1 (\Omega)$ the sequence $\Green (f_n)$ is equiintegrable. In particular, there exists a weakly convergent subsequence $\Green (f_{n_k}) \rightharpoonup u$ in $L^1 (\Omega)$.
\end{theorem}

\subsection{Estimate of $Vu$ in $L^1 (\Omega)$}

In order to have an extension to an $L^1$ theory we introduce the following estimate
\begin{theorem}
	Let $V \in L^\infty_+ (\Omega)$ and $u = \Green_V (f)$. Then, for every $K \Subset \Omega$
	\begin{equation}
	\label{eq:estimate norm Vu}
	\int_ \Omega V |u| \le \| \Green(1) \|_{L^\infty(\Omega)}  \left(     \frac{1}{\inf_K \Green(1)} + \| V \|_{L^\infty(K)}     \right)   \int_\Omega |f|.
	\end{equation}
\end{theorem}

\begin{proof}
	Assume first that $f \ge 0$. Then, we use $\psi = \Green (1) $ as a test function we deduce
	\begin{equation}
	\int_ \Omega u + \int_ \Omega V u \Green(1)  = \int_ \Omega f \Green (1).
	\end{equation}
	Clearly $\Green (1) |_K \ge c > 0$. Hence
	\begin{equation}
	\int_K Vu \le \frac 1 c \int_K Vu \Green(1) \le \frac 1 c \int_\Omega f \Green(1) \le C \int_ \Omega f.
	\end{equation}
	On the other hand,
	\begin{equation}
	\int_{ \Omega \setminus K } Vu \Green(1) \le C \| V \|_{L^\infty (\Omega \setminus K)} \int_ \Omega u \le C \int_ \Omega f.
	\end{equation}
	Thus,
	\begin{equation}
	\int_ \Omega V u \le C \int_ \Omega f. \\
	\end{equation}
	If $f$ changes sign we decompose as $f = f_+ - f_-$. We apply the result above for $u_1 = \Green_V (f_+), u_2 = \Green_V (f_-)$. Thus, $u = u_1 - u_2$, and so $|u| \le u_1 + u_2$. Hence,
	\begin{equation}
	\int_\Omega V|u| \le \int_ \Omega V u_1 + \int_ \Omega Vu_2 \le C \int_ \Omega f_+ + C \int_ \Omega f_- = C \int_ \Omega |f|.
	\end{equation}
	This completes the proof.
\end{proof}

\section{Uniqueness for general $V \ge 0$ }

\label{sec:uniqueness of Schrodinger}
\begin{theorem}
	Assume $|\{ V = +\infty \}| = 0$. There exists, at most, one solution $u \in L^1 (\Omega)$ of  \eqref{eq:fixed point formulation}.
\end{theorem}

\begin{proof} Let $u_1, u_2 \in L^1 (\Omega)$ be two solution. Then $u = u_1 - u_2 \in L^1 (\Omega)$ is a solution of $u = - \Green (Vu)$.
	
	For $k \in \mathbb N$ we define $V_k = V \wedge k \in L^\infty_+ (\Omega)$.
	We write
	\begin{equation}
		u = \Green (  (V_k - V) u  - V_k u ) = \Green (f_k - V_k u)
	\end{equation}
	where $f_k = (V_k - V) u \in L^1 (\Omega)$. Hence, due to \Cref{thm:uniqueness fixed point}, $u$ is the unique solution of $u + \Green (V_k u) = \Green (f_k)$ and we deduce that
	\begin{equation}
		\| u \|_{L^1 (\Omega)} \le C \| f _ k \|_{L^1 (\Omega)}.
	\end{equation}
	
	On the other hand, we have that $|(V-V_k) u | \le |V - V_k|  |u| \le V|u| \in L^1 (\Omega)$. Since the $V_k \to V$ a.e. we deduce that $(V-V_k)u \to 0$ a.e. in $\Omega$. Thus, due the Dominated Convergence Theorem we have $(V-V_k) u \to 0$ in $L^1 (\Omega)$ and so $u = 0$.
\end{proof}


\section{Existence for $(f,V) \in L^1 (\Omega) \times L^1_+ (\Omega)$}
\label{sec.exist.L1L1}

\begin{theorem}
	\label{thm:existence when f and V in L1}
	If $(f,V) \in L^1 (\Omega) \times L^1_+ (\Omega)$, there exists a solution.
\end{theorem}

The proof of this replicates the double limit argument in our previous paper \cite{diaz+g-c+vazquez2018} for more general operators.

\begin{lemma}[Monotonicity]
	If $V_1 \le V_2$ and $f_1 \ge f_2$ then $\Green_{V_1} (f_1) \ge \Green_{V_2} (f_2)$.
\end{lemma}
\begin{proof}
	\textbf{Step 1. Assume $f_2 \ge 0$.}
	Let $u_i \ge 0$ be the unique solutions of $u_i = \Green (f_i - V_i u_i)$. Let $w = u_1 - u_2$. It satisfies
	\begin{equation}
		w + \Green (V_1 w) = \Green (f_1 - f_2 + (V_2 - V_1) u_2 ).
	\end{equation}
	Letting $F= f_1 - f_2 + (V_2 - V_1) u_2 \ge 0$ we have that $w$ is the unique solution of $w = \Green(F - V_1 w )$, and therefore $w \ge 0$ and, hence $u_1 \ge u_2$.
	
	\textbf{Step 2. $f_2$ has no sign.} The we decompose in positive and negative part $f_i = (f_i)_+ - (f_i)_-$. It is clear that
	\begin{equation}
		(f_1)_+ \ge (f_2)_+ \ge 0 , \qquad 0 \le (f_1)_- \le (f_2)_-.
	\end{equation}
	Applying the previous step we have
	\begin{equation}
	\Green((f_1)_+) \ge \Green((f_2)_+) , \qquad \Green((f_1)_- ) \le \Green ((f_2)_-).
	\end{equation}
	Therefore
	\begin{equation}
		\Green(f_1) = \Green ((f_1)_+ - (f_1)_-) \ge \Green ((f_2)_+ - (f_2)_-) = \Green(f_2).
	\end{equation}
	This completes the proof.
\end{proof}

\begin{proof}[Proof of \Cref{thm:existence when f and V in L1}]
	\textbf{Step 1.} $f \ge 0$. 	
	We define
	\begin{equation}
	V_k = V \wedge k , \qquad f_m = f \wedge m.
	\end{equation}
	We define $u_{k,m} = \Green_{V_k} (f_m) \in L^\infty (\Omega)$. Let $U_m = \Green (f_m) \in L^\infty (\Omega)$.
	
	\textbf{Step 1a. $k \to +\infty$.} Clearly $u_{k,m}$ is a non-increasing sequence on $k$ such that  $0 \le u_{k,m} \le U_m$, hence $u_{k,m} \to u_m$ in $L^1 (\Omega)$, due the Monotone Convergence Theorem. On the other hand
	\begin{equation}
	V_{k} u_{k,m} \le V U_m \in L^1 (\Omega)
	\end{equation}
	and we have
	\begin{equation}
	V_k u_{k,m} \to V u_m \qquad \textrm { a.e. } \Omega.
	\end{equation}
	Therefore, due the Dominated Convergence Theorem and, due to the estimate
	\begin{equation}
		\int_ \Omega V_k u_{k,m} \delta^\gamma \le C \int_ \Omega f_m \delta^\gamma
	\end{equation}
	we have that
	\begin{equation}
	V_k u_{k,m} \to V u_m \qquad L^1 (\Omega , {\delta^\gamma}).
	\end{equation}
	
	Hence
	\begin{equation}
	u_m = \lim_{k} u_{k,m} = \lim_{k} \Green (f_m - V_k u_{k,m}) = \Green (f_m - V u_m)
	\end{equation}
	and $u_m$ is the solution corresponding to $(f_m, V)$.
	
	\textbf{Step 1b. $m \to + \infty$.} The sequence $u_m$ is increasing. Since $\int_{ \Omega } u_m \le C$ due to the Monotone Convergence Theorem we have $u_m \to u$ in $L^1 (\Omega)$. Analogously  $V u_m \to Vu$ in $L^1 (\Omega,{\delta^\gamma})$. Furthermore ${u_m} = \Green (f_m - V_m u_m) \to \Green (f - Vu)$.
	
{\bf Step 2. \rm $f$ has no sign.} We decompose $f = f_+ - f_-$ and we apply Step 1.
\end{proof}

\section{Singular potential and measure data: CSOLAs}
\label{sec:existence for V singular at 0}

Once the theory of data $f$ and integrable potentials is complete, we address the novel question of measure data and possibly non-integrable potentials and the consequence of their interactions for the theory of existence.

\subsection{CSOLA: Limit of approximating sequences. Reduced measures}
 We regularize the potential by  putting
  \begin{equation}
 	V_\ee (x) = V(x) \wedge \frac 1 \ee.
 \end{equation}
 Since $V_\ee (x) \in L^\infty (\Omega)$, a Green kernel in the standard sense exists.

  For the remainder of this section we fix a measure $\mu \in \cM$. We want to understand what happens  to
 \begin{equation}
 	u_\ee = \Green_ {V_\ee} (\mu),
 \end{equation}
 i.e., to the solution of $Lu + V_\ee u=\mu$, as $\ee \to 0$.  We say that
 \begin{equation}
 	\ulim = \lim_{ \ee \to 0 } u_\ee
 \end{equation}
 is a Candidate Solution Obtained as Limit of Approximations (CSOLA). We will prove that such a convergence holds, at least, in $L^1 (\Omega)$. The main problem is to decide when the CSOLA is an actual dual solution.

 \label{sec:existence of SOLA}
 We prove the following
 \begin{theorem}
 	\label{thm:existence of CSOLA}
 	Assume that $V\ge0 $ satisfies condition \eqref{eq:V singular 0} from the introduction and let $\mu \ge 0$ be a nonnegative Radon measure. Then, there exist an integrable function $\ulim \ge 0$ and constants $(\alpha_\mu^x)_{x \in S} \in \mathbb R$ such that:
 	\begin{enumerate}[\rm i)]
 		\item \label{it:convergence uee}
 		$u_\ee \searrow \ulim$ in $L^1 (\Omega)$
 		
 		\item \label{it:convergence Vee uee away from 0}
 		$V_\ee u_\ee \to V \ulim$ in $L^1 (\Omega \setminus B_\rho({S}), {\delta^\gamma})$ for any $\rho > 0$
 		
 		\item \label{it:convergence Vee uee as measure}
 		$V_\ee u_\ee \rightharpoonup V \ulim + \sum_{x \in S} \alpha_\mu^x \Dirac_x$ weakly in $\cM (\Omega, {\delta^\gamma})$.
 		
 		\item  \label{it:problem for ulim} The limit satisfies the equation
 		\begin{equation}
 			\ulim + \Green (V \ulim) = \Green (\mu_r),
 		\end{equation}
 		where $\mu_r$ is the reduced measure
 		\begin{equation}
 			\mu_r = \mu - \sum_{x \in S} \alpha_ \mu ^x \Dirac_x.
 		\end{equation}
 	\end{enumerate}
 \end{theorem}

It is important to notice that, according to point (iv), $\ulim$ is the solution of \eqref{eq1} corresponding to the reduced measure $\mu_r$. We do not assert having solved  \eqref{eq1} with data $\mu$.

\begin{proof}[Proof of \Cref{thm:existence of CSOLA}]
	Let us prove \ref{it:convergence uee}). It is immediate that $u_\ee \ge 0$.
	Since the sequence $V_\ee$ is pointwise increasing, then the sequence $u_\ee$ is pointwise decreasing.  Thus, due the Monotone Convergence Theorem,  it has an $L^1(\Omega)$ limit, $\ulim\ge 0$.
	
	To prove \ref{it:convergence Vee uee away from 0}) we recall \eqref{eq:V singular 0} and thus $u_\ee \to \ulim$ in $L^1 (\Omega)$ is sufficient.
	
	To prove \ref{it:convergence Vee uee as measure}) we start by indicating that $V_\ee u_\ee \ge 0$.
	On the other hand, $\int_ \Omega V_\ee u_\ee {\delta^\gamma} \le C \int_{ \Omega }\delta^\gamma d\mu$. Indeed, taking
	let
		\begin{equation}
		K = \bigcup_{x \in S} \overline{ B_{\rho_x} (x)}
		\end{equation}
		where $0 < \rho _ x < \mathrm{dist}(S, \partial \Omega) / 2$ is small enough so that
		\begin{equation}
		\overline{ B_{\rho_x} (x)} \cap S = \{x\}.
		\end{equation}
	We have $S \subset \mathrm{int} (K)$. The estimate \eqref{eq:estimate norm Vu} is preserved.

	Thus, there exists a limit $\gamma \in \cM_+ (\Omega)$ as measures
	\begin{equation}
	V_\ee u _\ee \rightharpoonup \gamma \textrm{ in } \cM_+ (\Omega, {\delta^\gamma}).
	\end{equation}
	Due to the pointwise convergence away from $0$ the regular part of $\gamma$ is $V \underline u$. On the other hand, the singular support of $\mu$ is, at most, $S$. Thus, the singular part is a combination of $\Dirac$ measures. Hence,
	\begin{equation}
	\gamma =  \sum_{x \in S} \alpha_\mu^x \Dirac_x + V \ulim.
	\end{equation}
	Then,
	\begin{align}
		f - V_\ee u_\ee \rightharpoonup f  - \sum_{x \in S} \alpha_{ \mu }^x \Dirac_ x -  V \ulim \qquad   \textrm{ weak}-\star-\cM (\Omega,{\delta^\gamma}).
	\end{align}
	Hence,
	\begin{align}
		u_\ee = \Green \left( f - V_\ee u_\ee \right) \rightharpoonup \Green \left( f  - \sum_{x \in S} \alpha_{ \mu }^x \Dirac_ x -  V \ulim \right) \qquad   \textrm{ weakly } L^1 (\Omega).
	\end{align}
	Due to the uniqueness of limit,
	\begin{equation}
		\ulim = \Green \left( f  - \sum_{x \in S} \alpha_{ \mu }^x \Dirac_ x -  V \ulim \right)
	\end{equation}
	In other words.
	\begin{equation}
		\ulim + \Green (V \ulim ) = \Green \left( f  - \sum_{x \in S} \alpha_{ \mu }^x \Dirac_ x \right). \\
	\end{equation}
	This completes the proof.
\end{proof}

 \subsection{Every solution is a CSOLA}

\begin{proposition}
	\label{prop:if Green V exists it is the limit}
	Assume that a solution $u$ of \eqref{eq:fixed point formulation} exists, then $u_\ee \to u$ in $L^1 (\Omega)$.
\end{proposition}

\begin{proof}
	Let $u_\ee = \Green_{V_\ee} (\mu)$. Clearly $u = \Green_V (f) \ge 0$. By subtracting the two problems and letting $w_\ee = u_\ee - u$ and
	\begin{align}
	w_\ee &= \Green (\mu - V_\ee u_\ee ) - \Green(\mu - V u) \nonumber\\
	&= \Green (V u - V_\ee u_\ee ) \nonumber\\
	&= \Green ((V-V_\ee) u - V_\ee w_\ee)
	\end{align}
	Thus $w_\ee = \Green_{V_\ee} ((V-V_\ee)u)$.
	Since $V \ge V_\ee$ we have $(V-V_\ee)u \ge 0$. Also $(V - V_\ee)u \le V u \in L^1 (\Omega)$. Thus, by the Dominated Convergence Theorem and taking into account the pointwise limit we have
	\begin{equation}
	(V-V_\ee) u \to 0 \qquad L^1 (\Omega).
	\end{equation}
	Hence
	\begin{equation}
	0 \le w_\ee \le \Green ((V-V_\ee)u ) \to 0
	\end{equation}
	in $L^1 (\Omega)$.
\end{proof}

\subsection{CSOLAs are solutions if $\mu (S) = 0$. Solutions for $f \in L^1 (\Omega)$}

\begin{proposition}
	Let $\mu \ge 0$. Then,
	\begin{enumerate}[\rm i)]
			\item \label{it:scaling at 0}
		We can estimate the scaling at $0$ as
		\begin{equation}
		\label{eq:scaling estimate with respect to alpha mu}
		\lim_{\rho \to 0} \rho^{-2s} \int_{B_\rho (x)}  \ulim = c_x \Big(\mu(\{x\}) - \alpha_ \mu^x \Big).
		\end{equation}
		For some $c_x > 0$. In particular $\alpha_\mu^x \le \mu(\{x\})$.
		
		\item
		\label{it:existence if mu 0 is 0}
		If $\mu(\{x\}) = 0$ then $\alpha_{\mu}^x = 0$.
	\end{enumerate}
\end{proposition}

\begin{proof}
	We prove \cref{it:scaling at 0}). We rearrange the fact that $\ulim = \Green_V (\mu - \sum_{x \in S} \alpha_\mu^x \Dirac_x)$ as
	\begin{equation}
	\sum_{x \in S} \alpha_\mu^x  \Green (\psi) (x) = \int_\Omega \Green(\psi) \mathrm{d} \mu - \int_{\Omega} \ulim \psi - \int_{\Omega} V \ulim \Green(\psi) .
	\end{equation}
	We subtract $\mu( \{x_0\}) \Green (\psi) (x_0) $ to deduce
	\begin{equation*}
	\Big (\alpha_\mu - \mu(x_0) \Big) \Green (\psi) (x_0) + \sum_{x_0 \ne x \in S} \alpha_\mu^x  \Green (\psi) (x)   = \int_\Omega \Green(\psi) d(\mu - \mu(x_0)\Dirac_{x_0}) - \int_{\Omega} \ulim \psi - \int_{\Omega} V \ulim \Green( \psi ) .
	\end{equation*}
	Take $\psi = \Green (\one_{ B_ \rho (x_0)} ) \rho^{-2s}$ and we deduce, due to \eqref{eq:G one rho} that
	\begin{align}
	c(\alpha_\mu - \mu(x_0)) &= -\lim_{\rho \to 0} \rho^{-2s} \int_{B_\rho (x_0)}  \ulim \le 0.
	\end{align}
	where $c>0$. \\
	
	We now prove \cref{it:existence if mu 0 is 0}). If $\mu (x_0)=0$ we can apply \Cref{thm:local integral of G mu near support} to deduce
	\begin{equation}
	0 \le \frac 1 c \lim_{ \rho \to 0 } \rho^{-2s} \int_{B_\rho (x_0)}  \ulim \le  \frac 1 c \lim_{ \rho \to 0 } \rho^{-2s} \int_{B_\rho (x_0)}  \Green( \mu ) = 0.
	\end{equation}
	Combining this with \cref{it:scaling at 0} we deduce that $\alpha_ \mu^{x_0} = 0$.
\end{proof}

	\begin{corollary}
		\label{cor:CSOLA measures away from S}
		If $\mu (S) = 0$ then $  \Green_{V_\ee} (\mu) \to \ulim$ in $L^1(\Omega)$, where $\ulim$ satisfies $\ulim = \Green(\mu - V \ulim)$.
	\end{corollary}

\begin{theorem}
	Let $V$ satisfy \eqref{eq:V singular 0}. Then, for every $0 \le f \in L^1 (\Omega)$ there is a solution $\ulim \in L^1 (\Omega)$. It is the unique solution of \eqref{eq1}.
\end{theorem}

\subsection{The operator $\Green_V : L^1 (\Omega) \to L^1 (\Omega)$}

\begin{corollary}
	Let $V$ satisfy \eqref{eq:V singular 0}. Then, $\Green_V : L^1 (\Omega) \to L^q (\Omega)$ for all $q < Q(1)$ is linear and continuous and $\Green_V(f)$ is the unique dual solution of \eqref{eq1}.
\end{corollary}

\begin{proof}
	For $f \in L^1 (\Omega)$ it is clear that the measure $\mu = f \dx $ satisfies $\mu( 0 ) = 0$. In particular $\Green_V (f)$ is defined. Furthermore, due to the strong $L^1$ convergence we have that
	\begin{equation}
		\| \Green_V (f) \|_{L^1} = \lim_{ \ee \to  0} \| \Green_{V _\ee} (f) \|_{L^1} \le \| \Green (f) \|_{L^1} \le C \| f \|_{L^1 }.
	\end{equation}
	Linearity is trivial and the result is proven.
	
	In fact, since $| \Green_{V_\ee} (f)|  \le |\Green (f)|$ we have, for $ 1\le q < Q(1)$,
	\begin{equation}
		\| | \Green_{V_\ee} (f) \||_{L^q (\Omega)}  \le \| \Green (f) \|_{L^q (\Omega) }.
	\end{equation}
	For $q>1$ we have weak $L^q(\Omega)$ compactness, and hence
	\begin{equation}
				\| \Green_V (f) \|_{L^q} \le \lim_{ \ee \to  0} \| \Green_{V _\ee} (f) \|_{L^q} \le \| \Green (f) \|_{L^q} \le C \| f \|_{L^1 }.
	\end{equation}
	This proves the result.
\end{proof}

\begin{remark}
	In fact, an $\Green_V : L^1 \to L^1$ theory can be constructed simply under the hypothesis $\Green : L^\infty \to L^\infty$. However, the aim of the paper is the study of measures.
\end{remark}

\section{Solvability. Characterization of the reduced measure}\label{sec.nonex}

We address now the cases where $\mu$ and $V$ are not compatible. We start by point masses.

\subsection{Concentration of measures when $\mu = \Dirac_x$. Possible non-existence}

When the measure $\mu$ is precisely a Dirac delta at $0$ we show that non-existence is due to a concentration of measure.
We remind the reader that we define the set $Z$ of incompatible points as
\begin{equation}
	Z = \{  x \in \Omega :  \textrm{ there is no solution of } u = \Green (\Dirac_x - Vu) \textrm{ such that } u, Vu \in L^1 (\Omega)   \}.
\end{equation}
\begin{theorem}
	\label{thm:concentration of measures}
	Assume \eqref{eq:V singular 0}. And let $u_\ee =\Green_{V_\ee} (\Dirac_x)$, i.e. solving $u_\ee = \Green(\Dirac_x - V_\ee u_\ee)$.
	\begin{equation}
		u_\ee = \Green_{V_\ee} (\Dirac_x) \searrow \ulim \quad \textrm{ where }
		\begin{dcases}
		 	\ulim \textrm{ is the unique solution of } u = \Green (\delta_x - Vu) & \textrm{if } x \notin Z, \\
		 	\ulim = 0 & \textrm{if }x \in Z.
		\end{dcases}
	\end{equation}	
	Furthermore, we have
	\begin{equation}
	V_\ee u_\ee \rightharpoonup
	\begin{dcases}
		V \ulim & \textrm{if }x \notin Z, \\
		\Dirac_x, & \textrm{if } x \in Z.
	\end{dcases}
\end{equation}
weak-$\star$ in $\cM (\Omega)$.
\end{theorem}

\begin{proof}[Proof of \Cref{thm:concentration of measures}]
(i)	If $x \notin S$ we apply \Cref{cor:CSOLA measures away from S} and we deduce that there is a solution of $u = \Green (\Dirac_x - Vu)$. Therefore $x \notin Z$.
	
(ii)	If $x \in S$ we know
	\begin{equation}
	V_\ee u_\ee \rightharpoonup V \ulim + \alpha_{ \Dirac_x }^x \delta_x.
	\end{equation}
	Since it will not lead to confusion, let us just use $\alpha =  \alpha_{ \Dirac_x }^x$. The reduced measure is
	\begin{equation}
		(\Dirac_x)_r = (1 - \alpha) \Dirac_x.
	\end{equation}

(iii)	If $\alpha = 1$ then $(\Dirac_x)_r = 0$, and so $\ulim = 0$. Clearly $\ulim = 0 \ne \Green (\delta_x) = \Green (\delta_x - V \ulim )$. By \Cref{prop:if Green V exists it is the limit} if there was a solution of $u = \Green (\delta_x - Vu)$, then $u = \ulim$, and so there is no solution.
	
(iv)	If $\alpha \ne 1$ we define
	\begin{equation}
		U := \frac{\ulim}{1- \alpha} =  \frac{1}{1-\alpha} \Green \left(   (\Dirac_x)_r  - V \ulim    \right)= \frac{1}{1-\alpha} \Green \left(   (1-\alpha) \Dirac_x   - V \ulim    \right) = \Green (\Dirac_x - V U).
	\end{equation}
	Hence, by \Cref{prop:if Green V exists it is the limit} we have $\underline u = U$ and, therefore, $ \alpha = 0$.
\end{proof}

\subsection{Characterization of the reduced measure}

We obtain an immediate consequence of the point mass analysis.

\begin{theorem}
	\label{thm:characterization of reduced measure}
	Assume \eqref{eq:V singular 0}. Then,
	\begin{equation}
	\label{eq:characterization of reduced measure}
	\mu_r = \mu - \sum_{x \in Z } \mu (\{x\}) \Dirac_x.
	\end{equation}
\end{theorem}

\begin{proof}
	By writing the decomposition
	\begin{equation}
		\mu = \mu - \sum_{x \in S} \mu (x) \delta_x + \sum_{x \in S \setminus Z} \mu (x) \delta_x + \sum_{x \in Z} \mu (x) \delta x.
	\end{equation}
	We solve the approximating problems by superposition
	\begin{equation}
			\Green_{V_\ee}	(\mu) = \Green_{V_\ee} \left( \mu - \sum_{x \in S} \mu (x) \delta_x \right) + \sum_{x \in S \setminus Z} \mu (x) \Green_{V_\ee} (\delta_x) + \sum_{x \in Z} \mu (x) \Green_{V_\ee} (\Dirac_x).
	\end{equation}
	We know from \Cref{thm:existence of CSOLA} we know that
	\begin{equation}
		\Green_{V_\ee} (\mu) \to \ulim, \qquad \ulim = \Green( \mu_r - V \ulim )
	\end{equation}
	Using \Cref{cor:CSOLA measures away from S} and \Cref{thm:concentration of measures} we deduce that
	\begin{align*}
		\Green_{V_\ee} \left( \mu - \sum_{x \in S} \mu (x) \delta_x \right) &\to u_1, \qquad u_1 = \Green \left(    \mu - \sum_{x \in S} \mu (x) \delta_x    - Vu_1     \right) \\
		\sum_{x \in S \setminus Z} \mu (x) \Green_{V_\ee} (\delta_x) &\to u_2, \qquad u_2 =   \Green \left(    \sum_{x \in S \setminus Z} \mu (x) \delta_x   - Vu_1     \right)\\
		 \sum_{x \in Z} \mu (x) \Green_{V_\ee} (\Dirac_x)& \to 0.
	\end{align*}
	Hence $\ulim = u_1 + u_2$ and we have that
	\begin{gather*}
		\Green(\mu_r - V\ulim ) =  \Green \left(    \mu - \sum_{x \in S} \mu (x) \delta_x    - Vu_1     \right) + \Green \left(    \sum_{x \in S \setminus Z} \mu (x) \delta_x   - Vu_1     \right)\\
		\Green( \mu_r - V \ulim  ) = \Green \left(    \mu - \sum_{x \in Z} \mu (x) \delta_x    - V(u_1 + u_2)     \right) \\
		\Green( \mu_r ) = \Green \left(    \mu - \sum_{x \in Z} \mu (x) \delta_x        \right)\\
		\sum_{x \in S} \alpha_x^\mu \Green (\delta_x) = \sum_{x \in Z} \mu (x) \Green (\delta_x).
	\end{gather*}
	Using the scaling in \Cref{thm:local integral of G mu near support} we deduce that
	\begin{equation}
		\alpha_x^\mu =
		\begin{dcases}
			\mu(x) & x \in Z, \\
			0 & x \notin Z.
		\end{dcases}
	\end{equation}
	This completes the proof.
\end{proof}

\subsection{Necessary and sufficient condition for existence of solution}
In this way we get the necessary and sufficient condition for existence of solution of \eqref{eq1}.

\begin{theorem}
	There exists a dual solution of \eqref{eq1} with data $\mu \in \cM(\Omega)$ if and only if $|\mu| (Z) = 0$.
\end{theorem}

\begin{proof}
	By \Cref{thm:existence of CSOLA} know that the CSOLA exists and it solves the problem with the reduced measure. By \Cref{prop:if Green V exists it is the limit}, if a solution exists it is the CSOLA. Therefore $\Green(\mu_r - Vu) = u = \Green (\mu - Vu)$. Hence $\Green(\mu) = \Green (\mu_r)$.
	Then, due to \Cref{thm:characterization of reduced measure} this implies that
	\begin{equation}
		\sum_{x \in Z} \mu(\{x\}) \Green(\Dirac_x) = 0.
	\end{equation}
	This is equivalent to $\mu(\{x\}) = 0$ for all $x \in Z$. Since $Z$ is countable, this is equivalent to $|\mu|(Z) = 0$.
\end{proof}

\section{Properties and representation of $\Green_V$}\label{sec:representation of G_V for general V}

\subsection{Extension of $\Green_V$. The CSOLA operator}
We can define the CSOLA operator, $\widetilde \Green_V$, which can be understood both as the limit of $\Green_{V_\ee}:\cM (\Omega) \to L^1 (\Omega)$ or as the extension of $\Green_V:L^1 (\Omega) \to L^1 (\Omega)$ to the space of measures:
\begin{equation}
\widetilde \Green_V (\mu) = \Green_V (\mu_r).
\end{equation}
\begin{remark}
	Notice that, due to \Cref{thm:concentration of measures},
	\begin{equation}
	\widetilde \Green_V ( \delta_{x} ) = \begin{dcases}
	\Green_V (\delta_x) & \textrm{ if } x \notin Z, \\
	0 & \textrm{ if } x \in Z .
	\end{dcases}
	\end{equation}
\end{remark}

\begin{theorem}
	The operator $\widetilde \Green_V : \cM (\Omega) \to L^1 (\Omega)$ is a linear continuous extension of $\Green_V$.
\end{theorem}

\begin{proof}
	In the proof of \Cref{thm:existence of CSOLA} it is easy to see that $\alpha_{\mu}$ is linear in $\mu$.
	
	 For $\mu \ge 0$
	\begin{equation}
	\int_ \Omega |\widetilde \Green_V (\mu)| = \int_ \Omega \Green_V (\mu - \alpha_\mu \Dirac_0) \le \int_ \Omega \Green (\mu - \alpha_\mu \Dirac_0) \le \int_ \Omega \Green (\mu) \le C \| \mu \|_{\mathcal M (\Omega)}.
	\end{equation}
	For $\mu$ general we repeat for the positive and negative parts to deduce
	\begin{equation}
	\| \widetilde \Green_V (\mu) \|_{L^1 (\Omega)} \le C \| \mu \|_{\cM (\Omega)}.
	\end{equation}
	This completes the proof.
\end{proof}

\begin{corollary}
	If $\mu_n \rightharpoonup \mu$ weakly in $\cM (\Omega)$ we have
	\begin{equation}
		\widetilde \Green_V (\mu_n) \rightharpoonup \widetilde \Green_V (\mu) \textrm{ in }L^1 (\Omega).
	\end{equation}
\end{corollary}

\begin{proposition}
	\label{prop:GV continuous weak star M to weak L1}
	If $\mu_n \rightharpoonup \mu$ weak-$\star$ in $\cM (\Omega)$ we have
	\begin{equation}
	\widetilde \Green_V (\mu_n) \rightharpoonup \widetilde \Green_V (\mu) \textrm{ in }L^1 (\Omega).
	\end{equation}
\end{proposition}

\begin{proof}
	Due to linearity we assume $\mu = 0$.
	
	\textbf{Step 1.} Assume $\mu_n \ge 0$ and $\mu = 0$. Let $\psi \in L^\infty (\Omega)$.
	\begin{align}
		0 \le \langle \widetilde \Green_V (\mu_{n}) , \psi_+ \rangle & \le \langle \Green (\mu_n) , \psi_+ \rangle \to 0 , \\
		0 \le \langle \widetilde \Green_V (\mu_n) ,  \psi_- \rangle & \le  \langle \Green (\mu_n) , \psi_- \rangle \to 0.
	\end{align}
	Thus
	\begin{equation}
		\langle \widetilde \Green_V (\mu_n) , \psi \rangle \to 0.
	\end{equation}	
	$\widetilde \Green_V (\mu_n) \rightharpoonup 0$ in $L^1 (\Omega)$.
	
	\textbf{Step 2.} Assume $\mu_n$ can change sign. The sequence $(\mu_n)_+$ and $(\mu_-)$ are bounded. Take a convergent subsequence of $(\mu_n)_+$ and, out of that subsequence, a convergent subsequence of $(\mu_n)_-$. Hence, there exist $\lambda_1, \lambda_2$ such that
	\begin{equation}
		(\mu_n)_+ \rightharpoonup \lambda_1, \qquad
		(\mu_n)_- \rightharpoonup \lambda_2
	\end{equation}
	By uniqueness of the limit $\mu = \lambda_1 - \lambda_2$. We apply the first part of the proof to deduce that the result.
\end{proof}

\subsection{Regularization $\Green_V : L^\infty (\Omega) \to \mathcal C(\overline \Omega)$ and kernel representation}

\begin{theorem}
	\label{thm:Green V L inf to continuous}
	$\Green_V : L^\infty (\Omega) \to \mathcal C(\overline \Omega)$ is continuous. Furthermore
	\begin{equation}
		{\Green_V (f)} (x) = \int_ \Omega \G_V (x,y) f(y) \quad \textrm{ where } \G_V (x,y) = \widetilde \Green_V (\delta_x) (y).
	\end{equation}
\end{theorem}

\begin{proof}
	For $f, \psi \in L^\infty (\Omega)$ we have
	\begin{equation}
	\int_ \Omega \Green_V (f) \psi = \int_ \Omega f \Green_V (\psi) = \int_ \Omega f \widetilde \Green_V (\psi)
	\end{equation}
	Let $\psi_\ee = \frac{1}{|B_\ee(x)|} \one_{B_\ee (x)} \to \Dirac_x$ for $x \in \Omega$. Then $\widetilde \Green_V (\psi) \to \widetilde \Green_V (\Dirac_x)$. Since $\Green_V (f)$ In particular
	\begin{equation}
	\widehat{\Green_V (f)} (x) = \int_ \Omega f \widetilde \Green_V (\Dirac_x).
	\end{equation}
	
	Let $x_n \to x$ in $\mathbb R^n$. We have that
	\begin{equation}
		\Dirac_{x_n} \rightharpoonup \Dirac_x \textrm{ weak}-{\star}-\cM (\Omega)
	\end{equation}
	Due to \Cref{prop:GV continuous weak star M to weak L1} we have
	\begin{equation}
		\widehat{\Green_V (f)} (x_n) = \int_ \Omega f \widetilde \Green_V (\Dirac_{x_n}) \to \int_ \Omega f \widetilde \Green_V (\Dirac_x) = \widehat{\Green_V (f)} (x).
	\end{equation}
	Hence $\widehat{\Green_V (f)}$ is continuous on $\bar \Omega$. We can express $\Green_V (f)$ as its precise representation.
\end{proof}

%
%
%

\subsection{The kernel $\G_V$ as limit of $\G_{V_\ee}$}
In this clear that $\G_{V_\ee} (x,y)$ is a pointwise non-increasing sequence. Thus there is a limit
\begin{equation}
\G_{V_\ee} \searrow \underline {\G_{V}} \qquad \textrm{ in } L^2 (\Omega \times \Omega).
\end{equation}
what we have proven in the previous section can be understood as follows:
\begin{equation}
\underline {\G_{V}} (0,y) = 0, \qquad \textrm { if } \widetilde G_V (\Dirac_0) =  0.
\end{equation}

But we know that $\Green_{V_\ee} (f) \to \Green_V (f)$ for $f \in L^1 (\Omega)$, therefore
\begin{equation}
\underline \G_V (x,y) = \G_V (x,y).
\end{equation}
Furthermore, since symmetry holds for $\G_{V_\ee}$, we give yet a further reason for the symmetry
\begin{equation}
	\G_V (x,y) = \G_V (y,x).
\end{equation}

\section{Characterization of $Z$. Maximum principle.}\label{sec.Z}

We first recall the results of Ponce and Orsina \cite{Orsina2018} about set $Z$ and
failure of the strong maximum principle for bounded data in the case $L=-\Delta$ and adapt it to our fractional setting. We then proceed with the actual characterization of $Z$
in our setting.

\subsection{Set of universal zeros. Failure of the strong maximum principle}
\label{sec:universal zero-set}

Ponce and Orsina  formalized the notion of set of universal zeros (or universal zero-set in their notation):
\begin{equation}
	Z_0 = \{ x \in \Omega : {\Green_V (f)} (x) = 0 \quad \forall f \in L^\infty (\Omega)  \}
\end{equation}
in the context $s = 1$. As noted in their paper this is a failure of the strong maximum principle.
For $\Ls = -\Delta$ in  \cite{Orsina2018},  the universal zero-set is characterized as
\begin{equation}
	\label{eq:universal zero-set terms of delta}
	Z_0 = Z.
\end{equation}
Furthermore, the authors show that $\Green_V(\mu)$ exists for $L=-\Delta$ if and only if $|\mu| (Z) = 0$.
This leads them to indicate that in $Z \ne \emptyset$ then \emph{the Green kernel does not exist}.  However, the authors do indicate that, when $|\mu|(Z) = 0$ then (in our notation) the unique solution is written
\begin{equation}
	\Green_V (\mu) (x) = \int_ \Omega \Green_V (\Dirac_x) (y) \mathrm{d} \mu (y).
\end{equation}

In order to connect these assertions with the results in \Cref{sec:representation of G_V for general V}, in this paragraph we prove the following:
\begin{theorem}

	\label{thm:Z V decomposed} Assume \eqref{eq:V singular 0} and \eqref{eq:G is symmetric}--\eqref{eq:regularization}. It holds that
	\begin{equation}
		\label{eq:kernel G V as Green V of delta x}
		\widetilde \Green_V ( \delta_x ) (y) = \G_V (y,x)
	\end{equation}
	Then, the following are equivalent
	\begin{enumerate}[\rm i)]
		\item \label{it:zero of Green V of delta x}
			$\widetilde \Green_V (\Dirac_x) = 0$ (i.e. $x \in Z$)
		
		\item \label{it:zero of GV of x}
			$\G_V (x, \cdot) = 0$ a.e. in $\Omega$.
		
		\item \label{it:no max principle}
			$\Green_V (f) (x) = 0$ for all $f \in L^\infty (\Omega)$.
			
		\item \label{it: zero of Green V of 1}
			$\Green_V (\one_\Omega) (x) = 0$.
	\end{enumerate}
\end{theorem}
\begin{proof}
	
	It is easy to see that
	\begin{equation}
		\widetilde \Green_V ( \delta_x ) (y) = \int_ \Omega \G_V (y,z) d\delta_x (z) = \G_V (y,x).
	\end{equation}
	
	We prove that: \ref{it:zero of Green V of delta x}  $\iff$ \ref{it:zero of GV of x} $\implies$ \ref{it:no max principle} $\implies$ \ref{it: zero of Green V of 1} $\implies$ \ref{it:zero of GV of x}.
	
	The equivalence between \cref{it:zero of Green V of delta x} and \cref{it:zero of GV of x} is immediate from \eqref{eq:kernel G V as Green V of delta x}.
	
	Assume that \cref{it:zero of GV of x}. Then, for $f \in L^\infty (\Omega)$ we have that
	\begin{equation}
		\Green_V (f) (x) = \int_ \Omega \G_V (x,y) f(y) dy = \int_ \Omega 0 f(y) dy = 0.
	\end{equation}
	This is precisely \cref{it:no max principle}.
	
	Since the function $\one_\Omega \in L^\infty (\Omega)$ clearly \cref{it:no max principle} implies \cref{it: zero of Green V of 1}.
	
	Assume \cref{it: zero of Green V of 1}. Then
	\begin{equation}
		0 = \Green_V (\one_ \Omega) (x) = \int_ \Omega \G_V (x,y) dy =  \int_ \Omega |\G_V (x,y)| dy
	\end{equation}
	Hence, \cref{it:zero of GV of x} holds. 	
\end{proof}

\subsection{Necessary and sufficient condition on $V$ so that $x \in Z$}

We now state and prove the final result that characterizes nonexistence in terms of the integrability of $V$.

\begin{theorem}	\label{thm:Z depending on V}
	Assume \eqref{eq:V singular 0}. Then
	\begin{equation}
		x \notin Z \iff V \Green( \delta_x) \in L^1 (B_\rho (x)) \textrm{ for some } \rho > 0.
	\end{equation}
	In particular, $Z \subset S$.
\end{theorem}

\begin{remark}
	Notice that
	\begin{equation}
		V \Green( \delta_x) \in  L^1 (B_\rho (x)) \iff \int_ {B_\rho (x)} \frac{ V(y) }{|x-y|^{n-2s}} dy < + \infty.
	\end{equation}
\end{remark}

\begin{proof} We may take $x=0$ for convenience. 	Let $U = \Green (\Dirac_0) \in L^1 (\Omega)$.

 (i) Assume first $V U \in L^1 (\Omega)$. Then, for the approximating sequence in \Cref{thm:existence of CSOLA} corresponding to $\mu = \Dirac_0$ we have
	\begin{equation}
	V_\ee u _\ee \le V U \in L^1 (B_\rho (x)).
	\end{equation}
	Thus, due to the Dominated Convergence Theorem we have
	\begin{equation}
	V_\ee u_\ee \to V \ulim \in L^1 (B_\rho (x)).
	\end{equation}
	Therefore, the same convergence holds in the sense of measures. In particular, $\alpha_\mu = 0$, and $\ulim$ satisfies \eqref{eq:fixed point formulation}. Therefore $\Green_V( \Dirac_0 )$ is defined and $0 \notin S$. Since $0 \notin Z_0$ we deduce $0 \notin Z_V$.

\medskip
	
(ii)	Conversely, assume $0 \notin Z$.
	Taking into account \eqref{eq:fixed point formulation} $f = \one_ \Omega$
	\begin{equation}
		\int_ \Omega V \Green_V (\one_ \Omega) \Green (\psi) \le \int_ \Omega \Green_V (\one_ \Omega) \psi + \int_ \Omega V \Green_V (\one_ \Omega) \Green (\psi) = \int_ \Omega  \Green (\psi), \qquad \forall 0 \le \psi \in L^\infty (\Omega).
	\end{equation}
	Since, by construction $V \Green_V (\one_ \Omega) \in L^1 (\Omega)$ we can take a sequence
	\begin{equation}
		0\le \psi_k = \frac{  \one _ { B_ {1/k} (x_0) } }  {|B_ {1/k} (x_0)|}.
	\end{equation}
	 Due to \Cref{lem:ae aprox of G delta 0}, $\Green (\psi_k) \to U$ a.e. in $\Omega$. Due to Fatou's lemma
	\begin{equation}
		\int_ \Omega V U \Green_V (\one_ \Omega) \le \int_ \Omega U.
	\end{equation}
	Towards a contradiction, assume that $\Green_V (\delta_0)$ is defined. Then $\Green_V (\one_ \Omega) (0) > 0$ and, due to \Cref{thm:Green V L inf to continuous} $\Green_V (\one_ \Omega) (0) \ge  c > 0$ on $B_\rho$ for some $\rho$. But then
	\begin{equation}
		c \int_ {B_\rho} V U  \le \int_ \Omega U < +\infty
	\end{equation}
	using that $U \in L^1 (\Omega)$. Since \eqref{eq:V singular 0} we have that
	\begin{equation}
		\int_{\Omega \setminus B_\rho} VU \le \| V \|_{L^\infty (\Omega \setminus B_\rho)} \int_ \Omega U < +\infty.
	\end{equation}
	Thus, $VU \in L^1 (\Omega)$.
\end{proof}


\section{Extensions and open problems}

The theory that has been developed in this paper can be extended in different directions.

\begin{itemize}

\item We may also treat the problems in space dimensions $n=1,2$ which, as is well known, are somewhat special for the standard Laplacian. Here,  there are some difficulties only in the case $n - 2s \le 0$ (which corresponds to $n = 1$ and $s \ge 1/2$, or $n = 2$ and $s = 1$) since, otherwise, the kernels have the same form. Thus, for for $n - 2s < 0$ the kernel is not singular at $x=y$ and, for $n = 2s$, it has a logarithmic singularity. In \cite{Bonforte+Vazquez2016}  the information on the estimates for the different typical operators is gathered, and some of the sources we cite include $n=1,2$ (see, for instance, Corollary 1.4 of \cite{KimKim2014}). Our computations can be adapted for these cases as it is done in the standard theory for the usual Laplace operator.

\item We may consider more general operators $\mathrm L$, like those considered in \eqref{eq:RFL} one can replace $|x-y|^{-(n+2s)}$ by a different kernel $\mathbb K (x,y)$ under some conditions. Furthermore, a similar logic applies for other spectral-type operators, like $(-\Delta + m I)_{\mathrm{SFL}}^s$.

\item We can replace the condition $f\in L^1(\Omega)$ by inclusion in a weighted space $f\in L^1(\Omega, w)$ like we did in \cite{diaz+g-c+vazquez2018},  where the optimal weight was
    $w=\mbox{dist}(x,\Omega^c)^s$. The weight depends on the operator.

\item  There is an interest in studying the interaction of singular potentials with diffuse measures. See, for instance,  \cite{Ponce2017} in the case of the classical Laplacian.

\item Problems with a combination of  linear and nonlinear zero-order terms, like
$$
	\mathrm L u + Vu = f(u).
$$

\item An interesting line is to consider the corresponding parabolic problems:
$$
	u_t+\mathrm Lu+Vu=f\,.
$$

\item 
Study of more general functions $V$. 
We will give a more detailed account of the following development.
It is natural to consider the case of $V \ge 0$ a Borel measurable function. Let us define a linear continuous operator
\begin{equation}
	\widetilde \Green_V : \cM (\Omega) \to L^1 (\Omega)
\end{equation}
given by
\begin{equation}
	\widetilde \Green_V (\mu) = \lim_{\ee \to 0} \Green_{V_\ee} (\mu).
\end{equation}
When a solution of \eqref{eq:fixed point formulation} exists, it is as before $\widetilde \Green_V (\mu)$.

This new operator is given by a kernel $\G_V$. Furthermore
\begin{equation}
	\G_{V_\ee}  \searrow \G_V \qquad \textrm { in }  L^1 (\Omega \times \Omega) = L^1 (\Omega; L^1 (\Omega)).
\end{equation}
We define the sets
\begin{equation}
	Z = \{x \in \Omega: \G_V (x,y) = 0 \textrm { for a.e. } y\in \Omega  \}.
\end{equation}

Given a measure $\mu$ we can split $\mu = \mu_{Z} + \underline \mu$ where
\begin{equation}
	\mu_{Z} (A) = \mu (A \cap Z), \qquad \underline \mu (A) = \mu (A \setminus Z).
\end{equation}

For $x_0 \in Z$ we have that $\widetilde \Green_V (\delta_{x_0}) = 0$, but is not a solution of \eqref{eq:fixed point formulation}, since $x_0 \notin Z_0$. Therefore $\Green_V (\delta_{x_0})$ does not exist.
Analogously, if $ \mu_Z \ne 0$, then $\Green (\mu)$ is not defined, and $\widetilde \Green_V (\mu) = 0$.

It remains to see that $\Green_V (\underline \mu)$ exists.

For a general $\mu$ we will have
\begin{equation}
V_\ee \Green_{V_\ee} (\mu) \rightharpoonup V \widetilde \Green_V (\mu) + \lambda_\mu.
\end{equation}
This new measure $\lambda_\mu$ may be complicated and have an strange support.
The expected result is
\begin{equation}
	\lambda_\mu = 0 \iff \mu_Z = 0.
\end{equation}
In the case $Z = \{0\}$, it holds that $\lambda_\mu = \mu_Z$ so this result might be maintained.

This is equivalent to the natural extension of the results in \cite{Orsina2018} and their result is
\begin{equation}
	\Green_V (\mu) \textrm{ is defined } \iff \mu (Z) = 0.
\end{equation}

\end{itemize}

\section*{Acknowledgements}
 The first author is funded by MTM2017-85449-P (Spain).
 The second author is partially funded by Project  MTM2014-52240-P (Spain). Performed while visiting at Univ.\ Complutense de Madrid.


\begin{thebibliography}{10}
	\expandafter\ifx\csname url\endcsname\relax
	\def\url#1{\texttt{#1}}\fi
	\expandafter\ifx\csname doi\endcsname\relax
	\def\doi#1{\burlalt{doi:#1}{http://dx.doi.org/#1}}\fi
	\expandafter\ifx\csname urlprefix\endcsname\relax\def\urlprefix{URL }\fi
	\expandafter\ifx\csname href\endcsname\relax
	\def\href#1#2{#2}\fi
	\expandafter\ifx\csname burlalt\endcsname\relax
	\def\burlalt#1#2{\href{#2}{#1}}\fi
	
	\bibitem{Brezis2003}
	P.~B{\'{e}}nilan and H.~Brezis.
	\newblock {Nonlinear problems related to the Thomas-Fermi equation}.
	\newblock {\em Journal of Evolution Equations}, 3(4):673--770, 2003.
	\newblock \doi{10.1007/s00028-003-0117-8}.
	
	\bibitem{Benilan+Brezis+Crandall1975}
	P.~B\'enilan, H.~Brezis, and M.~G. Crandall.
	\newblock {A semilinear equation in {$L^1(R^N)$}}.
	\newblock {\em Ann. Scuola Norm. Sup. Pisa Cl. Sci. (4)}, 2(4):523--555, 1975.
	
	\bibitem{bonforte+figalli+vazquez2018}
	M.~Bonforte, A.~Figalli, and J.~V{\'{a}}zquez.
	\newblock {Sharp boundary behaviour of solutions to semilinear nonlocal
		elliptic equations}.
	\newblock {\em Calculus of Variations and Partial Differential Equations},
	57(2):1--34, 2018.
	\newblock \doi{10.1007/s00526-018-1321-2}.
	
	\bibitem{Bonforte+Sire+Vazquez2015}
	M.~Bonforte, Y.~Sire, and J.~L. V{\'{a}}zquez.
	\newblock {Existence, uniqueness and asymptotic behaviour for fractional porous
		medium equations on bounded domains}.
	\newblock {\em Discrete and Continuous Dynamical Systems- Series A},
	35(12):5725--5767, 2015,
	\burlalt{arXiv:1404.6195}{http://arxiv.org/abs/arXiv:1404.6195}.
	\newblock \doi{10.3934/dcds.2015.35.5725}.
	
	\bibitem{Bonforte+Vazquez2016}
	M.~Bonforte and J.~L. V{\'{a}}zquez.
	\newblock {Fractional nonlinear degenerate diffusion equations on bounded
		domains part I. Existence, uniqueness and upper bounds}.
	\newblock {\em Nonlinear Analysis, Theory, Methods and Applications},
	131:363--398, 2016,
	\burlalt{arXiv:1508.07871}{http://arxiv.org/abs/arXiv:1508.07871}.
	\newblock \doi{10.1016/j.na.2015.10.005}.
	
	\bibitem{Brezis2004a}
	H.~Brezis, M.~Marcus, and A.~C. Ponce.
	\newblock {A new concept of reduced measure for nonlinear elliptic equations}.
	\newblock {\em Comptes Rendus Mathematique}, 339(3):169--174, 2004.
	\newblock \doi{10.1016/j.crma.2004.05.012}.
	
	\bibitem{brezis+marcus+ponce2007}
	H.~Brezis, M.~Marcus, and A.~C. Ponce.
	\newblock Nonlinear elliptic equations with measures revisited.
	\newblock {\em Mathematical Aspects of Nonlinear Dispersive Equations (J.
		Bourgain, C. Kenig, and S. Klainerman, eds.), Annals of Mathematics Studies},
	163:55--110, 2007.
	
	\bibitem{BucurValdi}
	C.~Bucur and E.~Valdinoci.
	\newblock {\em Nonlocal diffusion and applications}, volume~20 of {\em Lecture
		Notes of the Unione Matematica Italiana}.
	\newblock Springer, [Cham]; Unione Matematica Italiana, Bologna, 2016.
	\newblock \doi{10.1007/978-3-319-28739-3}.
	
	\bibitem{CabreSire2014}
	X.~Cabr\'{e} and Y.~Sire.
	\newblock Nonlinear equations for fractional {L}aplacians, {I}: {R}egularity,
	maximum principles, and {H}amiltonian estimates.
	\newblock {\em Ann. Inst. H. Poincar\'{e} Anal. Non Lin\'{e}aire},
	31(1):23--53, 2014.
	\newblock \doi{10.1016/j.anihpc.2013.02.001}.
	
	\bibitem{CaffSilv2007}
	L.~A. Caffarelli and L.~Silvestre.
	\newblock An extension problem related to the fractional {L}aplacian.
	\newblock {\em Comm. Partial Differential Equations}, 32(7-9):1245--1260, 2007.
	\newblock \doi{10.1080/03605300600987306}.
	
	\bibitem{Caffarelli+Stinga2016}
	L.~A. Caffarelli and P.~R. Stinga.
	\newblock {Fractional elliptic equations, Caccioppoli estimates and
		regularity}.
	\newblock {\em Annales de l'Institut Henri Poincare (C) Analyse Non Lineaire},
	33(3):767--807, 2016,
	\burlalt{arXiv:1409.7721}{http://arxiv.org/abs/arXiv:1409.7721}.
	\newblock \doi{10.1007/s00526-014-0815-9}.
	
	\bibitem{chen+veron2014}
	H.~Chen and L.~V{\'{e}}ron.
	\newblock {Semilinear fractional elliptic equations involving measures}.
	\newblock {\em Journal of Differential Equations}, 257(5):1457--1486, 2014,
	\burlalt{arXiv:1305.0945v2}{http://arxiv.org/abs/arXiv:1305.0945v2}.
	\newblock \doi{10.1016/j.jde.2014.05.012}.
	
	\bibitem{Cozzi2017}
	M.~Cozzi.
	\newblock {Interior regularity of solutions of non-local equations in Sobolev
		and Nikol'skii spaces}.
	\newblock {\em Annali di Matematica Pura ed Applicata}, 196(2):555--578, 2017,
	\burlalt{arXiv:1601.02819}{http://arxiv.org/abs/arXiv:1601.02819}.
	\newblock \doi{10.1007/s10231-016-0586-3}.
	
	\bibitem{DiNezza2012}
	E.~{Di Nezza}, G.~Palatucci, and E.~Valdinoci.
	\newblock {Hitchhiker's guide to the fractional Sobolev spaces}.
	\newblock {\em Bulletin des Sciences Mathematiques}, 136(5):521--573, 2012,
	\burlalt{arXiv:1104.4345}{http://arxiv.org/abs/arXiv:1104.4345}.
	\newblock \doi{10.1016/j.bulsci.2011.12.004}.
	
	\bibitem{diaz+gc+rakotoson+temam:2018veryweak}
	J.~I. D{\'{i}}az, D.~G{\'{o}}mez-Castro, J.-M. Rakotoson, and R.~Temam.
	\newblock {Linear diffusion with singular absorption potential and/or unbounded
		convective flow: The weighted space approach}.
	\newblock {\em Discrete and Continuous Dynamical Systems}, 38(2):509--546,
	2018, \burlalt{arXiv:1710.07048}{http://arxiv.org/abs/arXiv:1710.07048}.
	\newblock \doi{10.3934/dcds.2018023}.
	
	\bibitem{diaz+g-c+vazquez2018}
	J.~I. D{\'{i}}az, D.~G{\'{o}}mez-Castro, and J.~V{\'{a}}zquez.
	\newblock {The fractional Schr{\"{o}}dinger equation with general nonnegative
		potentials. The weighted space approach}.
	\newblock {\em Nonlinear Analysis}, pages 1--36, 2018,
	\burlalt{arXiv:1804.08398}{http://arxiv.org/abs/arXiv:1804.08398}.
	\newblock \doi{10.1016/j.na.2018.05.001}.
	
	\bibitem{Evans1998}
	L.~C. Evans.
	\newblock {\em {Partial Differential Equations}}.
	\newblock American Mathematical Society, Providence, Rhode Island, 1998.
	
	\bibitem{FelKassV2015}
	M.~Felsinger, M.~Kassmann, and P.~Voigt.
	\newblock The {D}irichlet problem for nonlocal operators.
	\newblock {\em Math. Z.}, 279(3-4):779--809, 2015.
	\newblock \doi{10.1007/s00209-014-1394-3}.
	
	\bibitem{Gilbarg+Trudinger2001}
	D.~Gilbarg and N.~S. Trudinger.
	\newblock {\em {Elliptic Partial Differential Equations of Second Order}}.
	\newblock Springer-Verlag, Berlin, 2001.
	
	\bibitem{Grubb2015}
	G.~Grubb.
	\newblock {Fractional Laplacians on domains, a development of H{\"{o}}rmander's
		theory of $\mu$-transmission pseudodifferential operators}.
	\newblock {\em Advances in Mathematics}, 268:478--528, 2015.
	\newblock \doi{10.1016/j.aim.2014.09.018}.
	
	\bibitem{KimKim2014}
	K.-Y. Kim and P.~Kim.
	\newblock Two-sided estimates for the transition densities of symmetric
	{M}arkov processes dominated by stable-like processes in {$C^{1,\eta}$} open
	sets.
	\newblock {\em Stochastic Process. Appl.}, 124(9):3055--3083, 2014.
	\newblock \doi{10.1016/j.spa.2014.04.004}.
	
	\bibitem{Kuusi+Mingione+Sire2015a}
	T.~Kuusi, G.~Mingione, and Y.~Sire.
	\newblock {\em {Nonlocal Equations with Measure Data}}, volume 337.
	\newblock 2015.
	\newblock \doi{10.1007/s00220-015-2356-2}.
	
	\bibitem{Orsina2018}
	L.~Orsina and A.~C. Ponce.
	\newblock {On the nonexistence of Green's function and failure of the strong
		maximum principle}.
	\newblock 2018,
	\burlalt{arXiv:1808.07267}{http://arxiv.org/abs/arXiv:1808.07267}.
	
	\bibitem{Ponce2016}
	A.~C. Ponce.
	\newblock {\em {Elliptic PDEs, Measures and Capacities From the Poisson
			Equation to Nonlinear Thomas-Fermi Problems}}.
	\newblock European Mathematical Society Publishing House, Zurich, 2016.
	
	\bibitem{Ponce2017}
	A.~C. Ponce and N.~Wilmet.
	\newblock {Schr{\"{o}}dinger operators involving singular potentials and
		measure data}.
	\newblock {\em Journal of Differential Equations}, 263(6):3581--3610, 2017.
	\newblock \doi{10.1016/j.jde.2017.04.039}.
	
	\bibitem{Rakotoson2018}
	J.~M. Rakotoson.
	\newblock {Potential capacity and applications}.
	\newblock pages 1--37, 2018,
	\burlalt{1812.04061}{http://arxiv.org/abs/1812.04061}.
	\newblock \urlprefix\url{http://arxiv.org/abs/1812.04061}.
	
	\bibitem{Ros-Oton2016}
	X.~Ros-Oton.
	\newblock {Nonlocal elliptic equations in bounded domains: a survey}.
	\newblock {\em Publicacions Matem{\`{a}}tiques}, 60:3--26, 2016.
	
	\bibitem{Ros-Oton2014}
	X.~Ros-Oton and J.~Serra.
	\newblock {The Dirichlet problem for the fractional Laplacian: Regularity up to
		the boundary}.
	\newblock {\em Journal des Mathematiques Pures et Appliquees}, 101(3):275--302,
	2014, \burlalt{arXiv:1207.5985}{http://arxiv.org/abs/arXiv:1207.5985}.
	\newblock \doi{10.1016/j.matpur.2013.06.003}.
	
	\bibitem{Triebel}
	H.~Triebel.
	\newblock {\em {Interpolation Theory, Function Spaces, Differential
			Operators}}.
	\newblock North-Holland, Amsterdam, 1978.
	
	\bibitem{vazquez_1983}
	J.~L. V\'azquez.
	\newblock {On a semilinear equation in $\mathbb R^2$ involving bounded
		measures}.
	\newblock {\em Proceedings of the Royal Society of Edinburgh: Section A
		Mathematics}, 95(3-4):181--202, 1983.
	\newblock \doi{10.1017/S0308210500012907}.
	
\end{thebibliography}
\end{document}